\documentclass[12pt]{amsart}
\usepackage{amssymb,amsfonts,latexsym,amsmath,amsthm,graphicx}
\usepackage{parskip}
\usepackage{mathrsfs}
\usepackage{mathtools} 
\mathtoolsset{showonlyrefs,showmanualtags}

\DeclareMathOperator{\sign}{sign}

\usepackage{hyperref}

\usepackage[margin=1.16in]{geometry}

\numberwithin{equation}{section}
\newtheorem{thm}{Theorem}[section]
\newtheorem{prop}[thm]{Proposition}

\newtheorem{lemma}[thm]{Lemma}

\newtheorem{prob}[thm]{Problem}
\theoremstyle{remark} 
\newtheorem{remark}[]{Remark}

\newtheorem*{definition}{Definition}

\newcommand{\be}{\begin{equation}}
\newcommand{\ee}{\end{equation}}

\newcommand{\EE}{\mathbb{E}}
\newcommand{\PP}{\mathbb{P}}
\newcommand{\RR}{\mathbb{R}}
\newcommand{\CC}{\mathbb{C}}

\newcommand{\DD}{\mathbb{D}}
\newcommand{\ph}{\hat{p}}
\newcommand{\qh}{\hat{q}}
\newcommand{\xh}{\hat{x}}
\newcommand{\yh}{\hat{y}} 
\newcommand{\mP}{\mathcal{P}}
\newcommand{\mH}{\mathcal{H}}
\newcommand{\cE}{\mathcal{E}}

\newcommand{\sA}{\mathscr{A}}
\newcommand{\sF}{\mathscr{F}}

\newcommand{\sP}{\mathscr{P}}
\newcommand{\sK}{\mathscr{K}}

\newcommand{\xdot}{\dot{x}}
\newcommand{\ydot}{\dot{y}}
\newcommand{\udot}{\dot{u}}
\newcommand{\vdot}{\dot{v}}
\newcommand{\hq}{\hat{q}}
\newcommand{\hp}{\hat{p}}
\newcommand{\N}{\mathcal{N}}
\newcommand{\Nh}{\hat{N}}
\newcommand{\sT}{\mathscr{T}}

\usepackage[colorinlistoftodos]{todonotes}

\newcommand{\e}{\varepsilon}

\newcommand{\p}{\partial}

\begin{document}

\title[Limit cycle enumeration]{Limit cycle enumeration in random vector fields}
\date{}

\author[E. Lundberg]{Erik Lundberg}

\begin{abstract}
We study the number and distribution of the limit cycles of a planar vector field whose component functions are random polynomials.
We prove a lower bound on the average number of limit cycles when the random polynomials are sampled from the Kostlan-Shub-Smale ensemble.
Investigating a problem introduced by Brudnyi [Annals of Mathematics (2001)] we also consider a special local setting of counting limit cycles near a randomly perturbed center focus, and
when the perturbation has i.i.d. coefficients, we prove a limit law showing that the number of limit cycles situated within a disk of radius less than unity converges almost surely to the number of real zeros of a logarithmically-correlated random univariate power series.
We also consider infinitesimal perturbations where we obtain precise asymptotics on the global average count of limit cycles for a family of models.
The proofs of these results use novel combinations of techniques from dynamical systems and random analytic functions.
\end{abstract}


\maketitle

\section{Introduction}

The second part of Hilbert's sixteenth problem
asks for a study of the number and relative positions of the limit cycles
of a planar polynomial system of ordinary differential equations.
The problem was included in Smale's list of problems for the next century \cite{Smale1998},
where he stated that
``except for the Riemann hypothesis,
this seems to be the most elusive of Hilbert's problems.''
Although far from solved, this problem has attracted a great deal of attention and influenced several developments within the field of dynamical systems.
See \cite{Ilyashenko}, \cite{Li},
\cite{UribeHossein}, \cite{Roussarie}, and \cite{Caubergh} for surveys.

The finiteness of the number of limit cycles for any polynomial system was shown by Ilyashenko \cite{Ilya} and independently \'Ecalle \cite{Ecalle}.
The problem of proving the existence of an upper bound depending only on the degree of the polynomial system remains open;
even in the case of quadratic vector fields no uniform upper bound is known.
Relevant progress in bifurcation theory \cite{IlyaYakPolycycle}, \cite{Kaloshin} established bounds on the number of limit cycles bifurcating from elementary polycycles, addressing a local finiteness problem considered to be of fundamental importance in the global problem.
Several steps of progress 
\cite{Chen},
\cite{Khovanski}, \cite{Varnchenko},
\cite{Petrov}, \cite{IlyaYakoInvent}, \cite{Gavrilov},
\cite{BinNovYak},
\cite{BinDor}
have been made on establishing uniform quantitative bounds for a special  ``infinitesimal'' Hilbert's sixteenth problem that was proposed by Arnold \cite{ArnoldProb}
and which is restricted to near Hamiltonian systems.
Concerning lower bounds, in \cite{ChLl} it is shown that there are degree-$d$ polynomial systems whose number of limit cycles grows as $d^2 \log d$ with $d$ (cf. \cite{HanLi}).
A number of papers consider limit cycle enumeration problems for particular classes of interest, such as Li\'enard systems \cite{DumPanRou}, \cite{Llibre},
quadratic systems \cite{Bautin}, \cite{Gavrilov}, systems on a cylinder \cite{LN} which are 
related to systems with homogeneous nonlinearities \cite{CalRuf},
and systems arising in control theory
\cite{LeoKuz}.

The current work concerns the
following probabilistic perspective on 
enumeration of limit cycles of planar systems associated to polynomial vector fields.

\begin{prob}\label{prob:main}
Study the number
and distribution in the plane (including their relative positions) of the limit cycles of a vector field whose component functions are random polynomials.
\end{prob}

This probabilistic point of view was introduced by A. Brudnyi in \cite{Brudnyi1}, \cite{Brudnyi2}
where he developed powerful general results on the distribution of zeros of analytic functions depending analytically on a multi-variate parameter and used those estimates to establish an upper bound for the number of small amplitude limit cycles near a randomly perturbed center focus.  While the study of random polynomials and random analytic functions has seen several breakthroughs over the past twenty years, there has been a lack of further progress on Problem \ref{prob:main} since the (2001, 2003) work \cite{Brudnyi1, Brudnyi2} mentioned above.  

In the current paper, we establish three new types of results within the setting of Problem \ref{prob:main}.  Most relevant to the previous work \cite{Brudnyi1}, \cite{Brudnyi2}, we establish a probabilistic \emph{limit law} for a perturbative model closely related to the one introduced by Brudnyi, where we show that the limit cycle counting statistic converges almost surely to the number of real zeros of a certain random univariate power series, see Theorem \ref{thm:as} in Section \ref{sec:Brudnyi}.  This is the first instance of a probabilistic limit law in the setting of Problem \ref{prob:main}. A second novel direction introduced in this paper is the probabilistic study of \emph{bifurcating} limit cycles for infinitesimal perturbations where we are able to obtain precise asymptotics which we present in Section \ref{sec:infinitesimalpert} below.
A third novel direction of the current paper is the first probabilistic study of a non-perturbative model where we establish a lower bound on the expected number of limit cycles, see Section \ref{sec:KSS}.


We remark in passing that the \emph{first} part of Hilbert's sixteenth problem concerns the topology of real algebraic manifolds (see the above mentioned survey \cite{Li} that includes discussion of both the first and second part of Hilbert's sixteenth problem). 
This occupies a separate setting from the second part, namely, real algebraic geometry as opposed to dynamical systems. Yet, the current work builds on some of the insights from recent studies on the topology of random real algebraic hypersurfaces 
\cite{NazarovSodin1} \cite{GaWe0}, \cite{sarnak}, \cite{GaWe2}, \cite{LLstatistics}, \cite{GaWe1}, \cite{GaWe3}, \cite{NazarovSodin2}, \cite{SarnakWigman}, \cite{FLL}, \cite{LerarioStecconi}.

\subsection{Limit cycles for the Kostlan-Shub-Smale ensemble}\label{sec:KSS}

Let us begin by considering a non-perturbative problem of estimating the global number of limit cycles when the vector field components are random polynomials
sampled from the so-called Kostlan-Shub-Smale ensemble.
\begin{equation}\label{eq:affine}
p(x,y) = \sum_{0 \leq j + k \leq d } 
a_{j,k}
\sqrt{\frac{d!}{(d-j - k)! j! k!}}  x^{j}  y^{k}, \quad a_{j,k} \sim N(0,1), \text{ i.i.d.}
\end{equation}
Among the Gaussian models of random polynomials, the Kostlan model is distinguished as the unique model that uses the monomials as a basis and is invariant under change of coordinates by orthogonal transformations of projective space $\RR \PP^2$ (there are other Gaussian models with this invariance but their description requires Legendre polynomials).  Moreover, the complex analog of the Kostlan model, obtained by taking independent complex Gaussians $a_{j,k} \sim N_{\CC}(0,1)$, is the only unitarily invariant Gaussian model of complex random polynomials.  For these reasons, the Kostlan-Shub-Smale model has become a model of choice in studies of random multivariate polynomials.

Kostlan \cite{kostlan:93} adapted Kac's univariate method \cite{kac43} to the study of zero sets of multivariate polynomials, 
and Shub and Smale \cite{Bez2} further showed that the 
average number of real solutions to a random system 
of $n$ equations in $n$ unknowns where the polynomials have degrees
$d_1,...,d_n$ equals $\sqrt{d_1 \cdots d_n}$,
which is the square root 
of the maximum possible number of zeros as determined by Bezout's theorem.

In particular, a planar vector field with random polynomial components of degree $d$
has $\sqrt{d^2}=d$ many equilibria on average.
We show that the average number of limit cycles as well grows (at least) linearly in the degree.

\begin{thm}[Lower bound for average number of limit cycles]
\label{thm:main}
Let $p,q$ be random polynomials of degree $d$ sampled independently
from the Kostlan ensemble.
Let $N_d$ denote the number of limit cycles of the vector field 
$$F(x,y) = \binom{p(x,y)}{q(x,y)}.$$
There exists a constant $c_0>0$ such that
\begin{equation}\label{eq:LB}
\EE N_d \geq c_0 \cdot d ,
\end{equation}	
for all $d$.
\end{thm}

\begin{remark}
Addressing relative positions of limit cycles, it follows from the method of the proof of Theorem \ref{thm:main}, which localizes the problem to small disjoint (and unnested) annuli,
that the same lower bound holds while restricting to the number $\hat{N}_d$ of \emph{empty} limit cycles (we refer to a limit cycle as ``empty'' if it does not surround any other limit cycle).
Together with the above-mentioned result of Kostlan-Shub-Smale
this determines the growth rate of $\EE \hat{N}_d$ to be linear in $d$; for all sufficiently large $d$ we have
\begin{equation}\label{eq:empty}
c_0 d \leq \EE \hat{N}_d \leq d .
\end{equation}
Indeed, each empty limit cycle contains an equilibrium point in its interior and distinct empty limit cycles have disjoint interiors.
\end{remark}

The estimates \eqref{eq:empty}
suggest the existence of a constant $c>0$ such that
$\EE \hat{N}_d \sim c \cdot d $ as $d \rightarrow \infty$.
We prove an analogous asymptotic result for a related model where the vector field components are random real analytic functions  sampled from 
the Gaussian ensemble induced by the Bargmann-Fock inner product.
This model arises as a rescaling limit of the Kostlan-Shub-Smale model \cite{BeGa}.

\begin{thm}[asymptotic for number of empty limit cycles]\label{thm:law}
Let $f,g$ be random real-analytic functions sampled independently
from the Gaussian space induced by the Bargmann-Fock inner product.
Let $\hat{N}_R$ denote the number of empty limit cycles situated within the disk of radius $R$ of the vector field 
$$F(x,y) = \binom{f(x,y)}{g(x,y)}.$$
There exists a constant $c>0$ such that
\begin{equation}\label{eq:law}
\EE \hat{N}_R \sim c \cdot R^2, \quad \text{as } R \rightarrow \infty.
\end{equation}
\end{thm}

The proof of Theorem \ref{thm:main}
uses transverse annuli and an adaptation of the ``barrier construction'' originally developed by Nazarov and Sodin for the study of nodal sets of random eigenfunctions \cite{NazarovSodin1}.
The proof of Theorem \ref{thm:law}
is based on yet another tool from the study of random nodal sets, the integral geometry sandwich from \cite{NazarovSodin2}, which reduces the current problem to controlling long limit cycles.

It is of interest, but seems very difficult in the above non-perturbative settings, to obtain asymptotic results or even upper bounds for the average total number of limit cycles (without restricting to empty limit cycles).
In the next sections we discuss this direction in special perturbative settings, beginning with the above-mentioned problem studied by Brudnyi \cite{Brudnyi1}.

\subsection{Limit cycles surrounding a perturbed center focus}\label{sec:Brudnyi}

In \cite{Brudnyi1}, Brudnyi considered
the limit cycles situated in the disk $\DD_{1/2}$ of radius $1/2$ centered at the origin for the random vector field
\begin{equation}\label{eq:Brudnyi}
F(x,y) = \binom{y + \e p(x,y)}{-x + \e q(x,y)}
\end{equation}
where 
\be\label{eq:p}
 p(x,y) = \sum_{1 \leq j+k \leq d} a_{j,k} x^j y^k
\ee
and
\be\label{eq:q}
q(x,y) = \sum_{1 \leq j+k \leq d} b_{j,k} x^j y^k
\ee
are random polynomials with the vector of coefficients sampled uniformly from the $d(d+3)$-dimensional Euclidean unit ball $\displaystyle \left\{ \sum_{1 \leq j+k \leq d} (a_{j,k})^2 + (b_{j,k})^2 \leq 1 \right\}$,
and where $\e=\e(d) =\frac{1}{40\pi\sqrt{d}}$.


Using complexification and pluripotential theory, Brudnyi showed \cite{Brudnyi1} that, with $F$ as in \eqref{eq:Brudnyi}, the average number of limit cycles of the system $\binom{\dot{x}}{\dot{y}} = F(x,y)$ residing within the disk $\DD_{1/2}$ is $O((\log d)^2)$,
which was later improved to $O(\log d)$ in \cite{Brudnyi2}.
We conjecture that this can be further improved to $O(1)$ based on the result below
showing almost sure convergence of the limit cycle counting statistic for a slightly modified model,
where the vector of coefficients is sampled uniformly from the cube $[-1,1]^{d(d+3)}$ rather than from the unit ball of dimension $d(d+3)$.
While we suspect sampling from the cube to lead to similar asymptotics as sampling from the unit ball, a natural reason to prefer the cube is that it represents independent sampling of coefficients.
Note that we relax the smallness of the perturbation allowing $\e = \e(d) \rightarrow 0$ at an arbitrary rate as $d \rightarrow \infty$.

\begin{thm}[limit law for the perturbed center focus]\label{thm:as}
Let $p,q$ be random polynomials of degree $d$ with coefficients sampled uniformly and independently from $[-1,1]$, and suppose $\rho < 1$.
Suppose $\e = \e(d) \rightarrow 0$ as $d \rightarrow \infty$.
Let $N_d(\rho)$ denote the number of limit cycles situated within the disk $\DD_{\rho}$ of the vector field
\begin{equation}\label{eq:perturbedcenter}
F(x,y) = \binom{y + \e p(x,y)}{-x + \e q(x,y)}.
\end{equation}
Then $N_d(\rho)$
converges almost surely (as $d \rightarrow \infty$) to
a non-negative random variable $X(\rho)$, satisfying $\EE X(\rho)< \infty$, that counts the number of real zeros in $(0,\rho)$ of a random univariate power series
$\sA_\infty(r)= \displaystyle \sum_{m=0}^\infty \zeta_m r^{2m+2},$
whose coefficients $\zeta_m$ are independent random variables with  mean zero 
and variance satisfying $\EE \zeta_m^2 \sim 8\pi/m$ as $m \rightarrow \infty$.
\end{thm}

\begin{remark}
As we show in the proof of Theorem \ref{thm:as}, the limiting random variable $N_d(\rho)$ counts the number of \emph{bifurcating limit cycles} of the limiting vector field whose components are the power series obtained by letting the degree $d$ tend to infinity. 
To elaborate, the proof of Theorem \ref{thm:as} shows that the original limit can essentially be replaced with the iterated limit, taking $d \rightarrow \infty$ before $\e \rightarrow 0$ (and this is why the outcome is independent of the particular rate at which $\e(d) \rightarrow 0$ as $d \rightarrow \infty$).  This ultimately reduces the problem to studying an \emph{infinitesimal} perturbation involving random bivariate power series, and classical perturbation theory of the Poincar\'e first return map then gives rise to the random univariate power series $\sA_\infty(r)$ referenced in the statement of the theorem.
\end{remark}

The detailed description of the random series $\sA_\infty(r)$ is given in  \eqref{eq:Ainfinity} below.  This type of random series, where the variance of the $m$th coefficient is asymptotically proportional to $m^{-1}$, is sometimes referred to as being \emph{logarithmically-correlated} since the two-point correlation function $\EE \sA_\infty(r) \sA_\infty(t)$ approximately agrees (in its tail) with the logarithmic series
 $$-8\pi (rt)^2 \log(1-(rt)^2) = \sum_{m=1}^\infty \frac{8\pi}{m} (rt)^{2m+2}.$$


We also note that $\sA_\infty$ falls into the class of random power series with coefficients of power law variance, a general class that has been considered by H. Flasche and Z. Kabluchko in \cite{Kabluchko}, where the authors investigate the asymptotic behavior of the expected number of zeros in the interval $(0,\rho)$ as $\rho \rightarrow 1^-$.  However, the particular power law rate $m^{-1}$ of the logarithmically-correlated case falls just at the edge of the cases that were studied in \cite{Kabluchko}.  We encounter a similar situation for random polynomials in Section \ref{sec:infinitesimalpert} below (see the comments after Theorem \ref{thm:powerlaw}), but the polynomial case has now been addressed in \cite{KLN} (see the statement below added in press).

It is of interest to study the behavior of $X(\rho)$ as the radius $\rho$ approaches unity (as in the above-mentioned work \cite{Kabluchko}), since $\rho=1$ is almost surely the radius of convergence of the series $\sA_\infty$ (and hence it is natural to expect $\EE X(\rho)$ to diverge as $\rho \rightarrow 1^-$).
We conjecture that the precise asymptotic behavior of $\EE X(\rho)$
is given by \be\label{eq:conj}
\displaystyle \EE X(\rho) \sim \frac{1}{\pi} \sqrt{ - \log (1-\sqrt{\rho}) }, \quad \text{as } \rho \rightarrow 1^-.
\ee
This conjecture is supported by the following heuristic which rests on an assumption that Tao and Vu's universality method (establishing a ``replacement principle'') \cite{TaoVu} can be applied in the current setting: If the random series (see the function $\sA_\infty$ defined in the proof of Theorem \ref{thm:as} below) whose zeros are counted by $X(\rho)$ is replaced by an analogous Gaussian series (replacing each coefficient by a Gaussian random variable of the same variance)
then the resulting Gaussian series is closely related to the one considered in \cite[Thm. 3.4]{Flasche} where a precise asymptotic has been derived.
The validity of replacing a random series with a Gaussian one as $\rho \rightarrow 1^-$ seems plausible considering previous success in applying the universality method in related settings \cite{TaoVu}, \cite{DoVu}, \cite{Kabluchko}, \cite{KLN}.  

\begin{remark}
A normal distribution law was established in \cite[Thm. 1.7]{Brudnyi1} for the number of zeros of certain related families of analytic functions, but it follows from Theorem \ref{thm:as}, with some attention to the details of its proof,
that a normal distribution limit law does not hold for $N_d(\rho)$ when $\rho<1$. Indeed, $N_d(\rho)$ obeys a discrete limit law and converges, without rescaling,
to a nondegenerate random variable $X(\rho)$ with discrete support  $\{0,1,2,...\}$.
\end{remark}

The underlying reason for the almost sure convergence of the number of limit cycles is that, within this particular model of random polynomials, $p,q$ tend toward random bivariate power series convergent in the unit bidisk $\DD \times \DD \subset \CC^2$, in particular, convergent in the real disk $\DD_\rho$ for any $\rho<1$.

It is desirable to count limit cycles beyond this domain of convergence, i.e., in a disk with radius $\rho>1$.
Under an assumption that $\e(d)$ shrinks sufficiently fast, the methods of \cite{Brudnyi1} can be utilized to obtain the estimate $O((\log d)^2)$.  For convenience, to state this result we return to the context of Brudnyi's model,
sampling the vector of coefficients from the unit ball instead of a cube.

\begin{thm}\label{thm:addendum}
Let $p,q$ be random polynomials with the vector of coefficients sampled uniformly from the Euclidean unit ball, and suppose $\rho > 1$.
Suppose $0<\e = \e(d) \leq \frac{(2 \rho)^{-d(d+3)}}{40\pi \sqrt{d}} $.
The expectation of the number of limit cycles of the vector field
\begin{equation}
F(x,y) = \binom{y + \e p(x,y)}{-x + \e q(x,y)}
\end{equation}
situated within the disk $\DD_{\rho}$
is bounded from above by a constant times $(\log d)^2$.
\end{thm}

\subsection{Infinitesimal perturbations}\label{sec:infinitesimalpert}

Finally, let us consider
enumeration of limit cycles for a random system of the form \eqref{eq:Brudnyi}
while first taking the limit $\e \rightarrow 0$ and then $d \rightarrow \infty$. 
This is a more tractable problem of enumerating the limit cycles that arise when the perturbation is \emph{infinitesimal}.
For generic $p,q$, this problem reduces to studying the zeros of the first Melnikov function (also known as the Poincare-Pontryagin-Melnikov function).

We first consider the case when $p,q$ are
sampled from the Kostlan-Shub-Smale model,
where we obtain the following precise asymptotic valid for any radius $\rho>0$.

\begin{thm}[square root law for Kostlan perturbations]\label{thm:sqrt}
Let $p,q$ be random polynomials of degree $d$ sampled independently from the Kostlan ensemble.
For $\rho>0$ let $N_{d,\e}(\rho)$ denote the number of limit cycles situated in the disk $\DD_\rho$ of the vector field
\begin{equation}
F(x,y) = \binom{y+\e p(x,y)}{-x + \e q(x,y)}.
\end{equation}
Then as $\e \rightarrow 0$, $N_{d,\e}(\rho)$ converges almost surely to a random variable $N_d(\rho)$ whose expectation satisfies the following asymptotic in $d$.
\begin{equation}\label{eq:arctan}
\EE N_d(\rho) \sim \frac{\arctan \rho}{\pi}\sqrt{d} , \quad \text{as } d \rightarrow \infty.
\end{equation}
\end{thm}

The proof of this result identifies the $\e \rightarrow 0$ limit as the number of zeros of a Poincar\'e-Pontryagin-Melnikov integral (the first Melnikov function associated to the perturbation) which is a random Gaussian function in this setting. We apply the Kac-Rice formula to determine the average number of zeros, and we use Laplace's method in the asymptotic analysis as $d \rightarrow \infty$. It is worth noting that this technique extends more generally to the study of randomly perturbed Hamiltonian systems, see Lemmas \ref{lemma:Nd}, \ref{lemma:EK}, \ref{lemma:twopt}, and Section \ref{sec:abelian}.

We next consider a family of models where $p$ and $q$ have independent coefficients with mean zero and variances that have a power law relationship to the degree of the associated monomial.  We collect the variances as deterministic weights, $c_m$ depending on the degree $m$. Namely,
\be\label{eq:pweight}
 p(x,y) = \sum_{m=1}^d \sum_{j+k = m} c_m a_{j,k} x^j y^k,
\ee
and
\be\label{eq:qweight}
q(x,y) = \sum_{m=1}^d \sum_{j+k = m} c_m b_{j,k} x^j y^k,
\ee
with 
\be\label{eq:powerlawvar}
c_m^2 \sim m^\gamma, \quad \text{as } m \rightarrow \infty, \quad \gamma > 0
\ee
and $a_{j,k}$, $b_{j,k}$ are i.i.d. with mean zero, unit variance, and finite moments.

\begin{thm}[random coefficients with power law variance]\label{thm:powerlaw}
Let $p,q$ be random polynomials of degree $d$
as in \eqref{eq:pweight}, \eqref{eq:qweight} with independent coefficients having power law variance as in \eqref{eq:powerlawvar} with $\gamma>0$.
Let $N_{d,\e}$ denote the number of limit cycles
(throughout the entire plane) of the vector field
\begin{equation}
F(x,y) = \binom{y+\e p(x,y)}{-x + \e q(x,y)}.
\end{equation}
Then as $\e \rightarrow 0$, $N_{d,\e}$ converges almost surely to a random variable $N_d$ whose expectation satisfies the following asymptotic in $d$.
\be\label{eq:gammapos}
\EE N_d \sim \frac{1+\sqrt{\gamma}}{2\pi} \log d, \quad \text{as } d \rightarrow \infty.
\ee
\end{thm}

The proof uses Do, Nguyen, and Vu's results \cite{DoVu},
which concern real zeros of univariate random polynomials with coefficients of polynomial growth.  Unfortunately, this method does not extend to the case where $p,q$ have i.i.d. coefficients, i.e., when $\gamma=0$. We leave this case as an open problem that motivates advancing the theory of univariate random polynomials in a particular direction.

\subsection*{Added in press}
After a preprint of the current paper was posted on arxiv, M. Krishnapur, O. Nguyen, and the author have addressed in \cite{KLN} the above-stated open problem of determining asymptotics for the case when $\gamma=0$ (and more generally for $\gamma \leq 0$).
It follows from their results that for any $\gamma \leq 0$ the average number of bifurcating limit cycles satisfies
$\EE N_d \sim \frac{1}{2\pi} \log d$ as $d \rightarrow \infty.$
While this outcome joins up continuously (in $\gamma$) with the result \eqref{eq:gammapos}, we stress that there is a discontinuity when considering the bifurcating limit cycles restricted to the unit disk.  For $\gamma>0$ the growth of such is logarithmic, whereas for $\gamma=0$ the growth is $o(\log d)$, and for $\gamma<0$ the number of bifurcating limit cycles restricted to the unit disk is $O(1)$.  Hence, the case $\gamma=0$ should be considered as critical, see \cite{KLN}.


\subsection*{Outline of the paper}
In Section \ref{sec:prelim},
we review some preliminaries: 
the construction and basic properties of the Kostlan-Shub-Smale ensemble, the notion of a transverse annulus from dynamical systems, the Poincar\'e-Pontryagin-Melnikov integral from perturbation theory of Hamiltonian systems, and the Kac-Rice formula for the average number of real zeros of random functions.
In Section \ref{sec:main} we present the proofs of Theorems \ref{thm:main} and \ref{thm:law}; the anonymous referee has pointed out that the proof of Theorem \ref{sec:main} can be greatly simplified using results of A. Lerario and M. Stecconi \cite{LerarioStecconi} that provide a high ground point of view; details for this alternative approach are presented in Remark \ref{rmk:PropReferee} after we present our comparatively pedestrian proof in its original form (which we have retained since it is rather elementary and self-contained and has the benefit that the estimates involved could readily be made quantitative).
In Section \ref{sec:Brudnyi}
we present the proofs of Theorems \ref{thm:as}
and \ref{thm:addendum}.
In Section \ref{sec:infinitesimal}
we present the proofs of Theorems
and \ref{thm:sqrt} and \ref{thm:powerlaw}.  We conclude with some discussion of future directions and open problems in Section \ref{sec:concl}.

\subsection*{Acknowledgements}
The author thanks
Zakhar Kabluchko for
directing his attention toward
the above mentioned result from
Hendrik Flasche's dissertation \cite{Flasche}.
The author is also grateful to the anonymous referee for a thorough reading and for many insightful comments that improved the paper.
The author acknowledges support from the Simons Foundation (grant 712397).

\section{Preliminaries}\label{sec:prelim}

In this section we review
some basics of Gaussian random polynomials
and dynamical systems
that will be needed.

\subsection{The Kostlan-Shub-Smale ensemble of random polynomials}

Let $\mP_d$ denote the space of polynomials
of degree at most $d$ in two variables,
and recall that there is a natural
isomorphism (through homogenization) between $\mP_d$ and the space $\mH_d$ of homogeneous polynomials of degree $d$ in three variables.

A Gaussian ensemble can be specified by choosing a scalar product on 
$\mH_d$ in which case $f$ is sampled according to:
$$\text{Probability} (f\in A) = \frac{1}{v_{n,d}}\int_A e^{-\frac{\|f \|^2}{2}} dV(f), $$
where $v_{n,d}$  is the normalizing constant that makes this a probability density function,
and $dV$ is the volume form 
induced by the scalar product. 
The \emph{Kostlan-Shub-Smale ensemble}
(which we may simply refer to as the ``Kostlan ensemble''),
results from choosing as a scalar product 
the \emph{Fischer product}\footnote{This scalar product also goes by many other names,
such as the ``Bombieri product'' \cite[p. 122]{KhLu}.} defined as:
 $$\langle f, g \rangle_F = \frac{1}{d! \pi^{3}}\int_{\CC^{3}} f(x,y,z) \overline{g(x,y,z)} e^{-|(x,y,z)|^2} dx dy dz.$$
The monomials are orthogonal with respect to the Fischer product, and the weighted monomials $\binom{d}{\alpha}^{1/2}x^{\alpha_1}y^{\alpha_2}z^{\alpha_3}$ form an orthonormal basis.
Consequently, the following expression relates the Fischer norm of $f=\sum_{|\alpha|=d}f_\alpha x^{\alpha_1}y^{\alpha_2}z^{\alpha_3}$ with its coefficients in the monomial basis (see \cite[Equation (10)]{NewmanShapiro}):
\be
\|f\|_F=\left(\sum_{|\alpha|=d}|f_\alpha|^2 \binom{d}{\alpha}^{-1} \right)^{\frac{1}{2}}.\ee
Here we are using multi-index notation $\alpha=(\alpha_1,\alpha_2,\alpha_3)$,
with $|\alpha|:=\alpha_1+\alpha_2+\alpha_3$,
and $\binom{d}{\alpha} := \frac{d!}{\alpha_1! \alpha_2! \alpha_3!}$.

Having chosen a scalar product, 
we can build the random polynomial
$f$ as a linear combination, with independent Gaussian coefficients, 
using an orthonormal basis for the associated scalar product.
The weighted monomials $\sqrt{\binom{d}{\alpha}} x^{\alpha_1}y^{\alpha_2}z^{\alpha_3} $ form an orthonormal basis for the Fischer product. 
Thus, sampling $f$ from the Kostlan model  (here again $\alpha=(\alpha_1,\alpha_2,\alpha_3)$ is a multi-index),
we have
$$f(x,y,z)=\sum_{|\alpha|=d}\xi_\alpha \sqrt{\binom{d}{\alpha}}x^{\alpha_1}y^{\alpha_2}z^{\alpha_3}, \quad \xi_\alpha\sim N\left(0,1 \right), \text{ i.i.d.}$$ 

Restricting to the affine plane $z=1$,
and setting $p(x,y,1)=f(x,y,z)$, 
we arrive at \eqref{eq:affine}.

\begin{remark}\label{rmk:basis}
The monomial basis is the natural basis for this model, but we are free to write the expansion with i.i.d. coefficients in front of another basis as long as it is orthonormal with respect to the Fischer product.  This fact will be used in the proof of Theorem \ref{thm:main}.
\end{remark}

Among Gaussian models built using the monomials as an orthogonal basis, the Kostlan ensemble has the distinguished property of being the unique Gaussian ensemble that is rotationally invariant
(i.e., invariant under any orthogonal transformation of projective space).

The two-point correlation function (or covariance kernel) of $p$, defined as $K(v_1,v_2) := \EE p(v_1)p(v_2)$, where
$v_1=(x_1,y_1), v_2=(x_2,y_2)$,
satisfies
\be\label{eq:2ptK}
K(v_1,v_2) = (1+x_1x_2 + y_1 y_2)^d,
\ee
which can be seen by applying the multinomial formula after using $\EE \xi_\alpha \xi_\beta = \delta_{\alpha,\beta}$ to obtain
\be 
\EE p(v_1) p(v_2) = \sum_{|\alpha|=d} \binom{d}{\alpha} (x_1 x_2)^{\alpha_1}(y_1 y_2)^{\alpha_2}.
\ee

The homogeneous random polynomial $f$ is distributed the same as if composed with
an orthogonal transformation $T$ of projective space.
The random polynomial $p$ is not invariant under under the induced action on affine space, but if we multiply it by the factor 
$(1+x^2 + y^2)^{-d/2}$ then it is. We state this as a remark that will be used later.

\begin{remark}\label{rmk:invariant}
For $p \in \mP_d$ a random Kostlan polynomial of degree $d$,
the random function $r(x,y) = (1+x^2 + y^2)^{-d/2} p(x,y)$ is invariant under the action $r \mapsto s_T \cdot r \circ \psi \circ T \circ \psi^{-1}$ on affine space induced by an arbitrary orthogonal transformation $T \in SO(3)$ of projective space, where $\psi:S^2 \rightarrow \RR^2$
denotes central projection, and $s_T \equiv 1$ in the case $\deg p$ is even, and $s_T(x,y) = \sign (\psi\circ T \circ \psi^{-1})(x,y)$ in the case $\deg p$ is odd. 
Here we are defining $\psi^{-1}(x,y)$ by assigning a value on the upper hemisphere (the map $\psi:S^2 \rightarrow \RR^2$ is two-to-one). The factor $s_T$ corrects for the odd symmetry at antipodal points (for odd degree polynomials, even degree polynomials need no correction) and ensures that point evaluation of $s_T \cdot r \circ \psi \circ T \circ \psi^{-1}$ at $(x,y)$ is equivalent to $p_h \circ T$ evaluated at $\psi^{-1}(x,y)$, where $p_h$ denotes homogenization of $p$ defined as $p_h(x,y,z) = z^d p(\frac{x}{z},\frac{y}{z})$.  In the proof of Theorem \ref{thm:main}, we will restrict to points $(x,y)$ in the unit disk and only consider transformations $\psi \circ T \circ \psi^{-1}$ mapping the origin to a point in the unit disk.  In that setting we will always have $s_T = 1$, so this factor can be ignored.
\end{remark}


\subsection{Transverse annuli}
In the study of dynamical systems,
a fundamental and widely used notion is that of a ``trapping region''.
In particular, a so-called \emph{transverse annulus} is useful in implementations of the Poincare-Bendixson Theorem.

\begin{definition}
	Suppose $F$ is a planar vector field.
	An annular region $A$ with $C^1$-smooth boundary
	is called a \emph{transverse annulus} for $F$ if
	\begin{itemize}
	\item[1.] $F$ is transverse to the boundary of $A$ with
	$F$ pointing inward on both boundary components
	or outward on both boundary components, and
	\item[2.] $F$ has no equilibria in $A$.
	\end{itemize}
\end{definition}

An application of the Poincare-Bendixson Theorem \cite{Gu} shows that presence of a transverse annulus forces the occurrence of a periodic orbit, and with the additional assumption that $F$ is polynomial (or even analytic) such periodic orbits must be limit cycles.  This gives the following result that allows transverse annuli to be used to estimate the number of limit cycles.

\begin{prop}\label{prop:transann}
	Suppose $F$ is a planar vector field with polynomial (or more generally real-analytic) components.
	Any transverse annulus for $F$ contains at least one limit cycle of $F$.
\end{prop}

\begin{proof}
Let $A$ be a transverse annulus for $F$. Without loss of generality, assume $F$ points inward along both boundary components of $A$ (in the case $F$ points outward, one considers the backward trajectory in what follows).  Fixing an initial condition $(x(0),y(0))=(x_0,y_0)$ on the boundary of $A$, the forward trajectory $\{(x(t),y(t))\}_{t \geq 0}$ is contained in $A$ and, by the Poincar\'e-Bendixson Theorem, approaches a periodic orbit $\Gamma$.  We claim that this periodic orbit is a limit cycle, i.e., in a sufficiently small neighborhood of $\Gamma$ there are no additional periodic orbits. Indeed, suppose for the sake of contradiction that there are periodic orbits arbitrarily nearby.  Fixing a transverse segment, consider the Poincar\'e first return map which is locally defined and analytic \cite[Sec. 3.4]{Perko}. Subtracting the identity map gives an analytic function that must vanish identically since it has an accumulation of zeros (corresponding to the accumulation of periodic orbits).  This contradicts the fact that $\Gamma$ has a nonperiodic trajectory (namely, the one with initial condition $(x_0,y_0)$) converging to it from one side.
\end{proof}


\subsection{The Poincar\'e-Pontryagin-Melnikov integral}

Let us consider the perturbed Hamiltonian system
\be\label{eq:Hpert}
\begin{cases}
\dot{x} = \partial_y H(x,y) + \e p(x,y) \\
\dot{y} = -\partial_x H(x,y) + \e q(x,y)
\end{cases},
\ee
with $H,p,q$ polynomials.
For $\e=0$
the trajectories of \eqref{eq:Hpert}
follow (connected components of) the level curves of $H$.
For $\e>0$ small, the limit cycles of the system can be studied (in non-degenerate cases)
using the Poincar\'e-Pontryagin-Melnikov integral (first Melnikov function) given by
\be\label{eq:Abel}
\sA(h) = \int_{C_h} p \, dy - q \, dx,
\ee
where $C_h$ denotes a connected component of the level set $\{ (x,y): H(x,y) = h \}$.

The following classical result is fundamental in the study of limit cycles of perturbed Hamiltonian systems, see \cite[Sec. 2.1 of Part II]{ChLi} and \cite[Sec. 26]{IlyashenkoBook}.

\begin{thm}\label{thm:PPM}
Let $\sP_{\e}(h)$ be the Poincar\'e
first return map of the system \eqref{eq:Hpert} defined on some
segment transversal to the level curves of $H$, where $h \in (a,b)$ marks the value taken by $H$. 
Then, $\sP_\e(h) = h + \e \sA(h)  + \e^2 E(h, \e)$, as $\e \rightarrow 0$,
where $\sA$ denotes the integral \eqref{eq:Abel} and where $E(h,\e)$ is analytic and uniformly bounded for $\e$ sufficiently small and for $h$ in a compact neighbourhood of $(a,b)$.
\end{thm}

In particular, it follows that if the zeros of $\sA$ are isolated and non-degenerate,
then each zero of $\sA$ corresponds to a limit cycle of the system \eqref{eq:Hpert}, see \cite[Sec. 2.1 of Part II]{ChLi}.

\subsection{The Kac-Rice formula}

The following result,
in its various forms, is fundamental in the study of real zeros of random analytic functions.
We follow \cite[Thm. 3.2]{AzaisWscheborbook},
which will suit our needs,
but we mention in passing that the Kac-Rice formula also holds in multi-dimensional and non-Gaussian settings.

\begin{thm}[Kac-Rice formula]\label{thm:KR}
Let $f:I \rightarrow \RR$ be a random function with $I \subset \RR$ an interval.
Suppose the following conditions are satisfied
\begin{itemize}
    \item[(i)] $f$ is Gaussian,
    \item[(ii)] $f$ is almost surely of class $C^1$,
    \item[(iii)] For each $t \in I$, the random vector $(f(t),f'(t))$ has a nondegenerate distribution,
    \item[(iv)] Almost surely $f$ has no degenerate zeros.
\end{itemize}
Then 
\be\label{eq:KRformula}
\EE |\{t \in I:f(t) =0\} | = \int_I \int_\RR |B| \rho_t(0,B) dB dt,
\ee
where $ \rho_t(A,B)$ denotes the
joint probability density function
of $A=f(t)$ and $B=f'(t)$.
\end{thm}

As observed by Edelman and Kostlan \cite{EdelmanKostlan95}, We can express \eqref{eq:KRformula} in terms of the two-point correlation function $K(r,t):=\EE f(r) f(t)$, namely,
\be\label{eq:EKform}
\EE |\{t \in I:f(t) =0\} | = \frac{1}{\pi} \int_I \sqrt{\frac{\partial^2}{\partial t \partial r} \log K(r,t) \rvert_{r=t=\tau}} d\tau.
\ee
This identity is based on the fact that the
joint density $\rho_t(A,B)$ of the Gaussian pair $A=f(t), B=f'(t)$  can be expressed in terms of the covariance matrix of $f(t)$ and $f'(t)$ which in turn can be computed from evaluation along the diagonal $r=t$ of appropriate partial derivatives of $K(r,t)$.

\section{Proofs of Theorems \ref{thm:main} and \ref{thm:law}}\label{sec:main}

\begin{proof}[Proof of Theorem \ref{thm:main}]
It is enough to prove a lower bound of the form \eqref{eq:LB}
while restricting our attention to just those limit cycles that are contained in the unit disk $\DD = \{ (x,y) \in \RR^2 : x^2 + y^2 < 1 \}$.

Let $A_d = A_d(0,0) = \{ (x,y) \in \RR^2 : d^{-1/2} < |(x,y)| < 2d^{-1/2} \}$ denote the annulus centered at $(0,0)$ with inner radius $d^{-1/2}$ and outer radius $2d^{-1/2}$.
For $0<r<1$ and $0 \leq \theta < 2\pi$ let $A_d(r,\theta)$ denote the image of $A_d$ under the mapping $\psi \circ T_{r,\theta} \circ \psi^{-1} : \RR^2 \rightarrow \RR^2$,
where $\psi:S^2 \rightarrow \RR^2$
denotes central projection (with $\psi^{-1}$ chosen to map to the upper hemisphere as in Remark \ref{rmk:invariant})
and $T_{r,\theta} : S^2 \rightarrow S^2$
denotes the rotation of $S^2$
such that  $\psi \circ T_{r,\theta} \circ \psi^{-1} : \RR^2 \rightarrow \RR^2$
fixes the line through the origin with angle $\theta$ and maps the origin to $(r\cos \theta,r \sin \theta)$.

From elementary geometric considerations, we notice that $A_d(r,\theta)$ is an annulus with elliptical boundary components, each having angle of alignment $\theta$.
The two elliptical boundary components are asymptotically concentric and homothetic.
As $d \rightarrow \infty$
the center is located at 
$(r \cos \theta, r \sin \theta) + O(d^{-1})$.
The major semi-axis of the smaller ellipse equals $a d^{-1/2} + O(d^{-1})$, 
and the major semi-axis of the larger ellipse is $2a d^{-1/2} + O(d^{-1})$, with $a = \sqrt{1+r^2}$.  The minor semi-axes of the inner and outer boundary components are $d^{-1/2}$ and $2d^{-1/2}$ respectively.

For some $k>0$ independent of $d$, 
we can fit at least $k \cdot d$
many disjoint annuli $A_d(r_i,\theta_i)$ in the unit disk $\DD$.
Letting $I_j$ denote the indicator random variable
for the event $\cE_j$ that the $j$th annulus is transverse for $F$,
we note that if $\cE_j$ occurs then Proposition
\ref{prop:transann} guarantees that there is at least one limit cycle contained in the annulus associated to $\cE_j$.

Since the annuli are disjoint,
the limit cycles considered above are distinct (even though the events $I_j$ are dependent),
and this gives the following lower bound for the expectation of $N_d$
\begin{align}\label{eq:linearity}
\EE N_d &\geq \EE \sum  I_j \\
  &= \sum \EE I_j \\
&= \sum \PP \cE_j,
\end{align}
where $j$ ranges over an index set of size at least $k \cdot d$,
and we have used linearity of expectation
going from the first line to the second.

\begin{prop}\label{prop:LB}
There exists a constant $c_1>0$, independent of $d$, such that for any $0<r<1$ and $0 \leq \theta < 2\pi$,
the probability that $A_d(r,\theta)$ 
is a transverse annulus for $F$ is at least $c_1$.
\end{prop}

Before proving the proposition, let us first see how it is used to finish the proof of the theorem. 
We can apply Proposition \ref{prop:LB} to each of the events $\cE_j$ to get
$\PP \cE_j \geq c_1$ for all $j$.
Combining this with \eqref{eq:linearity},
we obtain
$$\EE N_d \geq c_1 \cdot k \cdot d.$$
Since $k>0$ and $c_1>0$ are each independent of $d$,
this proves Theorem \ref{thm:main}.
\end{proof}

It remains to prove Proposition \ref{prop:LB}.
We will need the following lemmas.
The first of these provides a sup-norm estimate and is proved using a standard application of Markov's inequality; the lemma alternatively follows as a corollary of \cite{LerarioStecconi} as pointed out by the referee, see Remark \ref{rmk:lemmafollowsfromLS} below for details.  Nevertheless, we also include the more elementary proof to make the paper more self-contained.

\begin{lemma}\label{lemma:supest}
Let $F = \binom{p}{q}$ be a random vector field with polynomial components $p,q$ sampled from the Kostlan ensemble of degree $d$.
Let $D_d = \DD_{3d^{-1/2}}$ denote the disk
of radius $3d^{-1/2}$ centered at the origin.
There exists a constant $C_0 > 0$ such that
$$ \PP \left\{ \sup_{(x,y) \in D_d} | F(x,y)| \geq C_0 \right\} < \frac{1}{3}.$$
\end{lemma}

\begin{remark}\label{rmk:lemmafollowsfromLS}
The anonymous referee has pointed out that this lemma can be seen as a consequence of \cite[Thm. 23]{LerarioStecconi}.  Namely, taking $A$ to be the open set
$$A:= \{f \in C^\infty(\RR^2,\RR^2): \sup_{(x,y) \in \DD_3} |f(x,y)| > C_0 \},$$
by point (2) of \cite[Thm. 23]{LerarioStecconi} we have
$$ \liminf_{d\rightarrow \infty} \PP \{X_\infty \in A\} \leq \liminf_{d \rightarrow \infty} \PP \left\{ F\left( \frac{\cdot}{\sqrt{d}} \right) \in A\right\}=\liminf_{d \rightarrow \infty}  \PP \{ \sup_{(x,y) \in D_d} |F(x,y)| > C_0 \},$$
where $X_\infty$ is the field defined in point (1) of \cite[Thm. 23]{LerarioStecconi}.
Then the conclusion of the lemma follows, since the excursion probability
$\PP\{ X_\infty \in A \}$ of the Gaussian field $X_\infty$ goes to zero as $C_0 \rightarrow +\infty$
(as follows from, for instance, \cite[Thm. 2.1.1]{AdlerTaylor}).
\end{remark}


\begin{proof}[Proof of Lemma \ref{lemma:supest}]
Define $\ph(\xh,\yh) = p(d^{-1/2}\xh,d^{-1/2}\yh)$, and 
$\qh(\xh,\yh) = q(d^{-1/2}\xh,d^{-1/2}\yh)$,
and note that $(x,y) \in D_d$
corresponds to $(\xh,\yh) \in \DD_2(0)$.
It suffices to show that
\begin{equation}\label{eq:twoevents}
    \PP \left\{ \| \ph \|_{\infty,\DD_2(0)} > C_0/2 \right\} < \frac{1}{6}, \quad \text{and} \quad \PP \left\{ \| \qh \|_{\infty,\DD_2(0)} > C_0/2 \right\} < \frac{1}{6},
\end{equation}
where we used the notation $\displaystyle \| \ph \|_{\infty,\DD_2(0)} := \sup_{(\xh,\yh) \in \DD_2(0)} \ph(\xh,\yh).$
Indeed, the event that $\displaystyle \| F\|_{\infty,D_d} > C_0$
is contained in the union of the
two events considered in \eqref{eq:twoevents}.
We only need to prove the first statement in \eqref{eq:twoevents} since $\ph$ and $\qh$ are identically distributed.
For $(\xh,\yh) \in \DD_2(0)$ we have

\begin{align}
|\ph (\xh,\yh)| &= \left| \sum_{ |\beta| \leq d } 
a_{\beta}
\sqrt{\frac{d!}{(d-|\beta|)!\beta_1!\beta_2 !}}  \frac{\xh^{\beta_1}  \yh^{\beta_2}}{d^{|\beta|/2}} \right| \\
&\leq \sum_{ |\beta| \leq d } 
|a_{\beta}|
\frac{2^{|\beta|}}{\sqrt{\beta_1!\beta_2 !}}.
\end{align}

This implies
\begin{align}
\EE \|\ph \|_{\infty,\DD_2(0)} 
&\leq \sum_{ |\beta| \geq 0 } 
\sqrt{\frac{2}{\pi}}
\frac{2^{|\beta|}}{\sqrt{\beta_1!\beta_2 !}} \\
&\leq \sqrt{\frac{2}{\pi}} \sum_{ k \geq 0 } 
\frac{k \, 2^{k}}{\sqrt{\lfloor k/2 \rfloor}!},
\end{align}
which is a convergent series. 
So, we have shown that 
$$\EE \|\ph \|_{\infty,\DD_2(0)} \leq M<\infty,$$
with $M>0$ independent of $d$.
Applying Markov's inequality we have
$$\PP \{ \|\ph \|_{\infty,\DD_2(0)} > C_0/2 \} \leq \frac{2 M}{C_0},$$
and \eqref{eq:twoevents} follows
when $C_0$ is chosen larger than $12M$.
\end{proof}

For $0 \leq r \leq 1$, write $a=\sqrt{1+r^2}$,
and consider the random vector field
\begin{equation}\label{eq:defBr}
B_r(x,y) = \binom{ \xi_0 b_1(x,y)}{ \eta_0 b_2(x,y) } ,
\end{equation}
where $\xi_0,\eta_0$
are independent standard normal random variables, and
$$b_1(x,y) = \frac{1}{1+a} \left( -d^{1/2}ay + \frac{d^{1/2}}{V_d}x(3-d^2 x^2-d^2y^2 )\right)$$
$$b_2(x,y) = \frac{1}{1+a} \left( d^{1/2}x + \frac{d^{1/2}}{V_d} a y (3 - d^{2} x^2 - d^{2} y^2 )\right),$$
where
\be\label{eq:Vd}
V_d = 3+\sqrt{\frac{2}{(1-1/d)(1-2/d)}} + \sqrt{\frac{6}{(1-1/d)(1-2/d)}}.
\ee

In the proof of Proposition \ref{prop:LB}, the vector field $B_r$ will appear (by way of a choice of a new orthonormal basis) as a component (the key component for verifying the event in the statement of the proposition) in a decomposition of the Kostlan random vector field.  The value of the constant $V_d$ ensures that the associated basis elements have unit Fischer norm (see Lemma \ref{lemma:normone} below).  To give some intuition for the above choice of $b_1,b_2$ let us consider the annulus $A_d(0,0)$ which corresponds to the case $a=1$ for simplicity and write $B_r(x,y)$ as
\begin{equation}
    B_r(x,y) = \frac{1}{2} \left( B_t(x,y) + \frac{B_n(x,y)}{V_d} \right), 
\end{equation}
where
$$B_t(x,y) = d^{1/2} \binom{-\xi_0 y}{\eta_0 x} , \quad B_n(x,y) =  d^{1/2} \left( 3 - d(x^2 + y^2) \right) \binom{\xi_0 x}{\eta_0 y}.$$

In the model case that $\xi_0,\eta_0$ fall inside a common interval far from zero, say, in an interval $(C_1-\e,C_1+\e)$ with $C_1 \gg 1 \gg \e > 0$, the annulus $A_d(0,0)$ is highly transverse for $B_r$ with the vector field $B_n$ ensuring the inward pointing condition (here note that the scalar function $3-d(x^2+y^2)$ is positive on the inner boundary and negative on the outer boundary) and the vector field $B_t$ preventing occurrence of equilibria in $A_d(0,0)$ while pointing in an approximately orthogonal direction to $B_n$ and hence not disturbing its inward pointing property.  Details are given in the proof of Lemma \ref{lemma:barrier} and in its application within the proof of Proposition \ref{prop:LB} below.  First, we prove the following more basic lemma verifying that the value of $V_d$ leads to $b_1,b_2$ having unit Fischer norm.

\begin{lemma}\label{lemma:normone}
Let $B_r$ be the vector field defined in \eqref{eq:defBr}, and let $h_1(x,y,z)$ and $h_2(x,y,z)$ denote the homogeneous polynomials of degree $d$ that coincide with the components $b_1$ and $b_2$ of $B_r$ on the affine plane $z=1$.
Then $h_1$ and $h_2$ each have unit Fischer norm.
\end{lemma}

\begin{proof}[Proof of Lemma \ref{lemma:normone}]
We first recall that the monomials are orthogonal with respect to the Fischer inner product,
and next we recall the norms of each of the monomials appearing in $B_r(x,y)$:
$$\| x z^{d-1} \|_F = \| y z^{d-1} \|_F = d^{-1/2}$$
$$\| x^3 z^{d-3} \|_F = \| y^3 z^{d-3} \|_F = \sqrt{\frac{6}{d(d-1)(d-2)}}$$
$$\| x^2y z^{d-3} \|_F = \| y^2x z^{d-3} \|_F = \sqrt{\frac{2}{d(d-1)(d-2)}}$$
Then letting $h_1(x,y,z)$ and $h_2(x,y,z)$ denote the homogeneous polynomials of degree $d$ that coincide with $b_1$ and $b_2$ (respectively) on the affine plane $z=1$, we find
$$ \|h_1(x,y,z) \|_F =  \frac{1}{1+a} \left( a+ \frac{1}{V_d}\left[3+ \frac{\sqrt{6}+\sqrt{2}}{\sqrt{(1-1/d)(1-2/d)}}\right]\right),$$
which is unity by the choice of $V_d$,
and similarly
$$\| h_2(x,y,z) \|_F = \frac{1}{1+a} \left( 1 + \frac{a}{V_d} \left[ 3+ \frac{\sqrt{6}+\sqrt{2}}{\sqrt{(1-1/d)(1-2/d)}}\right]\right) = 1.$$
\end{proof}

Let $T_r = \psi \circ  R_r \circ \psi^{-1}$
where $\psi : S^2 \rightarrow \RR^2$ denotes central projection,
and $R_r$ denotes the rotation of $S^2$ such that $\psi \circ  R_r \circ \psi^{-1}$ fixes the $x$-axis and maps the origin to $(r,0)$.

Let $\hat{n}_r$
denote the inward-pointing unit normal vector
on $\p A(r,0)$,
and let 
$\Nh_r = \hat{n}_r \circ T_r$ denote
its pullback
by the transformation $T_r$.
More explicitly, we can write
\begin{equation}\label{eq:defNr}
    \Nh_r(x,y) = \pm \frac{1}{\sqrt{x^2/a^2 + y^2}} \binom{x/a}{y} + O(d^{-1/2}),
\end{equation}
where the sign is ``$+$'' for the inner boundary component and ``$-$'' for the outer component.

\begin{lemma}\label{lemma:barrier}
Fix $C_0>0$.
For $0 \leq r \leq 1$, 
let $B_r$ and $\Nh_r$ be as defined above.
There exists a constant $c_2>0$ independent of both $d$ and $r$, such that there is probability at least $c_2$ of both of the following being satisfied.
\begin{itemize}
    \item[(i)] Everywhere along $\p A_d$, the scalar product $\langle B_r(x,y), \Nh_r(x,y) \rangle$  of $B_r$ and $\Nh_r$ satisfies $$\langle B_r(x,y), \Nh_r(x,y) \rangle \geq 2C_0 .$$
    \item[(ii)] We have $$\inf_{(x,y) \in A_d} \| B_r(x,y) \| \geq 2C_0.$$
\end{itemize}
\end{lemma}


\begin{remark}\label{rmk:perturb}
Note that, properties (i) and (ii) 
imply that, for any vector field $F_0$ satisfying $|F_0(x,y)| \leq C_0$ for all $(x,y) \in A_d$, the perturbation
$B_r+F_0$ of $B_r$ satisfies
$\langle B_r + F_0 , \Nh_r \rangle > 0$ on $\p A_d$, and we also have that $B_r + F_0$ does not vanish in $A_d$, i.e., $A_d$ is a transverse annulus for $B_r+F_0$.
\end{remark}

\begin{proof}[Proof of Lemma \ref{lemma:barrier}]
Fix $0\leq r \leq 1$.
Let us decompose $B_r(x,y)$ defined in \eqref{eq:defBr} as
\begin{equation}\label{eq:Bdecomp}
    B_r(x,y) = \frac{1}{1+a} \left( B_t(x,y) + \frac{B_n(x,y)}{V_d} \right),
\end{equation}
where
$$B_t(x,y) = d^{1/2} \binom{-\xi_0 a y}{\eta_0 x} ,$$
and
$$B_n(x,y) =  d^{1/2} \left( 3 - d(x^2 + y^2) \right) \binom{\xi_0 x}{\eta_0 a y}.$$

For $C_1 \gg 1 \gg \e > 0$ fixed,
consider the event $\cE$ that 
\begin{equation}\label{eq:interval}
\xi_0,\eta_0 \in (C_1 - \e, C_1 + \e).
\end{equation}
The probability $c_2$ of $\cE$ is positive and independent of $d$ and $r$.
We will show that $\cE$
implies each of the properties (i) and (ii).

First we rewrite $B_t$ and $B_n$ using the 
condition \eqref{eq:interval} defining the event $\cE$:
\begin{equation}\label{eq:Bcdecomp}
B_t(x,y) = C_1 d^{1/2} \binom{-a y}{x} + d^{1/2} \binom{- \gamma_1  a y}{ \gamma_2 x} ,
\end{equation}
and 
\begin{equation}\label{eq:Bgdecomp}
B_n(x,y) = C_1d^{1/2} \left( 3 - d(x^2 + y^2) \right) \binom{x}{a y} + d^{1/2} \left( 3 - d(x^2 + y^2) \right) \binom{\gamma_1 x}{ \gamma_2 a y},
\end{equation}
with $-\e < \gamma_i < \e$ for $i=1,2$.

Note that $B_t$ is approximately
orthogonal to the vector field $\Nh_r$ along $\p A_d$, and $B_n$ is approximately parallel to $\Nh_r$ along $\p A_d$.
This is the basic idea used in the estimates that follow.

We recall from \eqref{eq:defNr}
that $\Nh_r(x,y) = \pm \frac{1}{\sqrt{x^2/a^2 + y^2}} \binom{x/a}{y}+O(d^{-1/2})$ on $\p A_d$.

Since $\binom{x/a}{y}$ is orthogonal to $\binom{-a y}{x}$,
we have, using \eqref{eq:Bcdecomp} and \eqref{eq:defNr},
\begin{equation}\label{eq:BtNr}
\left| \langle B_t(x,y) , \Nh_r(x,y) \rangle \right| = d^{1/2} \left| \left(\gamma_2 - \gamma_1 \right)\frac{xy}{\sqrt{x^2/a^2+y^2}}\right| + O(d^{-1/2}), \quad \text{on } \p A_d.
\end{equation}
Applying the estimates
$$ d^{1/2} \frac{|xy|}{\sqrt{x^2/a^2+y^2}} \leq 2a < 4 , \quad (x,y) \in \p A_d,$$
$$ |\gamma_2- \gamma_1| < 2 \e,$$
in \eqref{eq:BtNr} we obtain, for all $d$ sufficiently large,
\begin{equation}\label{eq:Bt}
\left| \langle B_t(x,y) , \Nh_r(x,y) \rangle \right| < 8\e, \quad \text{on } \p A_d.
\end{equation}

Using \eqref{eq:Bgdecomp}, we obtain for $(x,y) \in \p A_d$
\begin{align}\label{eq:alignBnNr}
\langle B_n(x,y) , \Nh_r(x,y) \rangle &= \pm a \frac{d^{1/2} \left( 3 - d(x^2 + y^2)\right)}{\sqrt{x^2/a^2+y^2}} \left( C_1(x^2/a^2+y^2) +   \gamma_1 x^2/a^2 + \gamma_2 y^2 \right) \\
 &\geq \pm a \frac{d^{1/2} \left( 3 - d(x^2 + y^2)\right)}{\sqrt{x^2/a^2+y^2}} (x^2/a^2+y^2) 
 \left( C_1 -2\e \right)  \\
  &\geq \pm d^{1/2} \left( 3 - d(x^2 + y^2)\right)\sqrt{x^2+y^2} 
 \left( C_1 -2\e \right),
\end{align}
where the choice of $\pm$ sign is determined according to the component of $\partial A_d$ as in \eqref{eq:defNr}.
We have
$$\pm d^{1/2} \left( 3 - d(x^2 + y^2)\right)\sqrt{x^2+y^2} = 2, \quad \text{on } \p A_d, $$
which holds on both components of $\p A_d$.
This gives
\be\label{eq:Bn}
\langle B_n(x,y) , \Nh_r(x,y) \rangle \geq 2C_1 - 4 \e, \quad \text{on } \p A_d.
\ee

The two estimates \eqref{eq:Bt} and \eqref{eq:Bn} together give 
$$\langle B_r(x,y) , \Nh_r(x,y) \rangle \geq \frac{1}{1+a} \left( \frac{2C_1 - 4 \e}{V_d} -8\e \right),\quad \text{on } \p A_d.$$
Since $1+a<3$ and $V_d$ converges to a postive constant $3 + \sqrt{6}+\sqrt{2}$ as $d \rightarrow \infty$,
for an appropriate choice of $C_1 \gg \e > 0$ we have
$\frac{1}{1+a} \left( \frac{1}{V_d} (2C_1 - 4\e)  -8\e \right) > 2C_0$
for all $d$ sufficiently large.
This completes the verification that property (i)
holds.

Next we check property (ii).
Throughout the annulus $A_d$ we have
\begin{align}\label{eq:normlowerbound}
(1+a)^2 \| B_r \|^2 &= \| B_t \|^2 + \frac{1}{V_d^2}\|B_n \|^2 + \frac{2}{V_d} \langle B_t , B_n \rangle \\
&\geq \| B_t \|^2 + \frac{2}{V_d} \langle B_t , B_n \rangle \\
&\geq \| B_t \|^2 - | \langle B_t , B_n \rangle |,
\end{align}
where we have used in the last line that $V_d$ which is defined in \eqref{eq:Vd} is at least $3$. Next we estimate $| \langle B_t , B_n \rangle |$.
\begin{align}\label{eq:normlowerboundcont}
| \langle B_t , B_n \rangle | &=  4 a d C_1 |(\gamma_2-\gamma_1)xy|
+2ad |(\gamma_2^2-\gamma_1^2)xy| \\
&\leq 8 a C_1 |\gamma_1-\gamma_2|
+4a |\gamma_2^2-\gamma_1^2|\\
&\leq 16 a C_1 \e + 4a \e^2 \\
&\leq 32 C_1 \e + 8 \e^2,
\end{align}
Thus, 
\begin{align}
(1+a)^2 \| B_r \|^2 &\geq \| B_t \|^2 - | \langle B_t , B_n \rangle | \\
&\geq \| B_t \|^2 - 32 C_1 \e - 8 \e^2 \\
&\geq (C_1-\e)^2 - 32 C_1 \e - 8 \e^2,
\end{align}
which is larger than $(6C_0)^2$ for an appropriate choice of $C_1 \gg \e >0$.
So we have
$$ (1+a)^2\| B_r \|^2 \geq (6C_0)^2,$$
which shows that property (ii) is satisfied
since $1+a < 3$.
\end{proof}

\begin{proof}[Proof of Proposition \ref{prop:LB}]
Fix $0<r<1$ and $0\leq \theta < 2\pi$.
Applying a rotation of the $xy$-plane about the origin (and using the rotation invariance of the Kostlan ensemble), we may assume that $\theta=0$.

Let $\cE$ denote the event that
$A_d(r,0)$ is a transverse annulus for $F$,
i.e., $\cE$ is the event that the following two conditions both hold:

\begin{itemize}
	\item The vector field $F$ points into $A_d(r,0)$
	on both boundary components, i.e.,
	$\langle F , \hat{n} \rangle > 0$ at each point
	on $\partial A_d(r,0)$ where $\hat{n}$
	denotes the inward pointing normal vector.
	\item $F$ has no equilibria in $A_d(r,0)$.
\end{itemize}

Equivalently, $\cE$ is the event that
$A_d(r,0)$ is a transverse annulus for $G$ defined as
$$G(x,y) = (1+x^2+y^2)^{-d/2} F(x,y).$$
By Remark \ref{rmk:invariant}, 
the component functions of $G$ are invariant (as random functions) under the group of transformations induced by
orthogonal transformations of projective space.

Let $T_r = \psi \circ  R_r \circ \psi^{-1}$,
where $\psi$ denotes central projection, and $R_r$ is the rotation such that $T_r$ fixes the $x$-axis and maps the origin to $(r,0)$.

As defined above, let $\Nh_r = \hat{n}_r \circ T_r$ denote the pullback to $\p A_d$ of the inward-pointing unit normal vector on $\p A_d(r,0)$ by the map $T_r$, where recall $A_d = A_d(0,0)$.
For each $(x,y) \in \p A_d$,
we recall from \eqref{eq:defNr}
that 
$\Nh_r(x,y) = \pm \frac{1}{\sqrt{x^2/a^2 + y^2}}\binom{x/a}{y} + O(d^{-1/2})$.

The event $\cE$ occurs if
$\langle G \circ T_r, \Nh_r \rangle > 0$ 
at each point on $\p A_d$
and $G \circ T_r$ has no equilibria in $A_d$.
We may replace $G \circ T_r$ by $G$
without changing the probability of the event
by Remark \ref{rmk:invariant}.
Since $G$ is a nonvanishing scalar multiple of $F$,
we may then replace $G$ by $F$.  
This leads us to consider the event $\cE_0$ that both of the following are satisfied
\begin{itemize}
	\item $\langle F , \Nh_r \rangle > 0$
	on $\partial A_d$.
	\item $F$ has no equilibria in $A_d$.
\end{itemize}
This event has the same probability as $\cE$.

As stated in Remark \ref{rmk:basis}, in the description of the Kostlan polynomial as a random linear combination with Gaussian coefficients,
one is free to choose the basis (as long as it is orthonormal with respect to the Fischer product).
Since the degree-$d$ homogenizations of the components of $B_r$ each have unit Fischer norm (by Lemma \ref{lemma:normone}), each can be used as elements in an orthonormal basis (orthonormal with respect to the Fischer product)
while expanding the components of the random vector field $F$ as a linear combination with i.i.d. Gaussian coefficients.
We write this in an abbreviated form as
\begin{align}
F(x,y) &= \binom{p(x,y)}{q(x,y)} \\
&= B_r(x,y) + F_r^\perp(x,y)\\
&= \binom{\xi_0 b_1(x,y)}{\eta_0 b_2(x,y)} + \binom{f_1^\perp(x,y)}{f_2^\perp(x,y)},
\end{align}
where in the first component all the terms involving the basis elements besides those in $b_1$ are collected in $f_1^\perp$, and in the second component all the terms involving the basis elements besides those in $b_2$ are collected in $f_2^\perp$.

Let 
$$\tilde{B}_r(x,y) = \binom{\xi_1 b_1(x,y)}{\eta_1 b_2(x,y)}$$ 
be an independent copy of $B_r(x,y)$,
i.e., $\xi_1$ and $\eta_1$ are standard normal random variables independent of eachother and of $\xi_0$ and $\eta_0$.
Define
$$F_r^{\pm} = F_r^\perp \pm \tilde{B}_r.$$
Then $F_r^{\pm}$ are each distributed as $F$,
and we can write
$$ F(x,y) = B_r(x,y) + \frac{1}{2}(F_r^+ + F_r^-) .$$

Define $\cE_1$ to be the event
described in Lemma \ref{lemma:barrier}
concerning $B_r(x,y)$.

Define $\cE_2$ to be the event that
$\|F_r^+(x,y)\|_\infty \leq C_0$,
and define $\cE_3$ to be the event that
$\|F_r^-(x,y)\|_\infty \leq C_0$,

If $\cE_1$, $\cE_2$, and $\cE_3$ all occur
then Lemma \ref{lemma:barrier} and Remark \ref{rmk:perturb} imply that $\cE_0$ occurs. Hence, we have
\be \label{eq:intersection}
\PP \{ \cE_0 \} \geq  \PP \{ \cE_1 \cap \cE_2 \cap \cE_3 \} .
\ee
By Lemma \ref{lemma:supest},
the complementary events $\cE_2^c$ and $\cE_3^c$
each have probability less than $1/3$.
Notice that $\cE_1$ is independent of $\cE_2$ and $\cE_3$, but $\cE_2$ and $\cE_3$ are not independent of each other.
Thus, we use a union bound with \eqref{eq:intersection} to estimate the probability of $\cE_0$
\begin{align}
\PP \{ \cE_0 \} &\geq  \PP \{ \cE_1 \cap \cE_2 \cap \cE_3 \} \\
&= \PP  \cE_1  \PP \{ \cE_2 \cap \cE_3 \} \\
&\geq \PP  \cE_1  \left( 1 - \PP \cE_2^c - \PP \cE_3^c \right)\\
&\geq (1/3)\PP  \cE_1  > 0,
\end{align}
which proves the proposition
since Lemma \ref{lemma:barrier} 
provides that the probability of $\cE_1$
is positive and independent of $d$.
This concludes the proof of the proposition.
\end{proof}

\begin{remark}\label{rmk:PropReferee}
The anonymous referee has pointed out that Proposition \ref{prop:LB} can be proved without the use of Lemmas \ref{lemma:supest}, \ref{lemma:normone}, 
 \ref{lemma:barrier} using the following very clean application of \cite[Thm. 23]{LerarioStecconi}, thus giving a high ground point of view on the result by understanding it as a consequence of the convergence of the Kostlan ensemble under rescaling.
 Namely, after having reduced the proof of the proposition to showing positivity of the probability of the event $\cE_0$, the proof can be finished as follows.
 Define $X_d(x,y) := F(\frac{1}{\sqrt{d}}(x,y))$.
 By point (2) of \cite[Thm. 23]{LerarioStecconi}, we have $X_d \rightarrow X_\infty$ in law in the space $C^\infty(\RR^2,\RR^2)$, where $X_\infty$ is the random vector field defined in point (1) of \cite[Thm. 23]{LerarioStecconi}.
 Fix $M>1>\e>0$, and consider the event $\cE_0^*$ such that
 \begin{itemize}
    \item $ \sup_{|(x,y)| \leq 2} |X_d(x,y)| \leq M$ 
	\item $\langle X_d(x,y) , (\frac{x^2}{a^2}+y^2)^{-1/2} \binom{x/a}{y} \rangle > \e$
	for all $(x,y)$ satisfying $|(x,y)| = 1$;
    \item  $\langle X_d(x,y) , -(\frac{x^2}{a^2}+y^2)^{-1/2} \binom{x/a}{y} \rangle > \e$
	for all $(x,y)$ satisfying $|(x,y)| = 2$;
	\item $X_d(x,y) \neq 0$ for all $(x,y)$ satisfying $1 \leq |(x,y)| \leq 2$.
\end{itemize}
The above are open conditions on $X_d$, and therefore it follows from point (2) of \cite[Thm. 23]{LerarioStecconi} that
$$ \liminf_{d \rightarrow + \infty} \PP \{X_d \in \cE_0^* \} \geq \PP\{X_\infty \in \cE_0^*\} = c.$$
Moreover, by point (3) of  \cite[Thm. 23]{LerarioStecconi}, we have $c>0$.
The conditions of the event $\cE_0$ can be described in terms of $X_d$ as follows (here $\hat{N}_r$ is defined as in \eqref{eq:defNr})
 \begin{itemize}
	\item $\langle X_d(x,y) , \hat{N}_r( \frac{1}{\sqrt{d}}(x,y)) \rangle > 0$
	for all $(x,y)$ satisfying $|(x,y)| \in \{1,2\}$;
	\item $X_d(x,y) \neq 0$ for all $(x,y)$ satisfying $1 \leq |(x,y)| \leq 2$.
\end{itemize}
From \eqref{eq:defNr} we have the following uniform convergence
\be \label{eq:unifN}
\hat{N}_r \left( \frac{1}{\sqrt{d}}(x,y) \right) \rightarrow \pm \left(\frac{x^2}{a^2}+y^2\right)^{-1/2} \binom{x/a}{y} + O(d^{-1/2}), \quad |(x,y)| \in \{1,2\},
\ee
from which it follows that, for $d$ sufficiently large, the event $\cE_0^*$ implies the event $\cE_0$.  
Indeed, the event $\cE_0^*$ together with the uniform convergence statement \eqref{eq:unifN} implies that for $|(x,y)| \in \{1,2\}$
$$ \left\langle X_d(x,y), \hat{N}_r \left( \frac{1}{\sqrt{d}}(x,y)\right) \right\rangle \geq \e - M \cdot O(d^{-1/2}),$$
which is positive for all $d$ sufficiently large.
We conclude that, for all $d$ sufficiently large, $\PP \{ X_d \in \cE_0 \} \geq c > 0$, as desired.
\end{remark}

\begin{proof}[Proof of Theorem \ref{thm:law}]
Let $f,g$ be random real-analytic functions sampled independently
from the Gaussian space induced by the Bargmann-Fock inner product
and normalized by the deterministic scalar $\exp\left\{ -(x^2+y^2)/2 \right\}$.
Explicitly,
$$f(x,y) = \exp\left\{ -(x^2+y^2)/2 \right\} \sum_{j,k \geq 0} a_{j,k} \frac{x^j y^k}{\sqrt{j! k!}} , \quad a_{j,k} \sim N(0,1).$$
Including the factor $\exp\left\{ -(x^2+y^2)/2 \right\}$ ensures that $f,g$ are invariant under translations \cite{BeGa}, and  multiplication of a vector field by a non-vanishing scalar function does not affect its trajectories, in particular, its limit cycles.

Let $\hat{N}_R$ denote the number of empty limit cycles situated within the disk of radius $R$ of the vector field 
$$F(x,y) = \binom{f(x,y)}{g(x,y)}.$$

We will use the integral geometry sandwich \cite{NazarovSodin1} 
\be\label{eq:IGS}
\int_{\DD_{R-r}} \frac{\N(x, r)}{|\DD_r |} dx
\leq \N(0,R) \leq \int_{\DD_{R+r}} \frac{\N^*(x, r)}{| \DD_r |} dx,
\ee 
where 
$\N(x,r)$ denotes the number of empty limit cycles completely contained in the disk $\DD_r(x)$ of radius $r$ centered at $x$, $\N^*(x,r)$ denotes the number of empty limit cycles that intersect $\overline{\DD_r(x)}$, and $|\DD_r|=\pi r^2$ denotes the area of $\DD_r$.
The statement of the integral geometry sandwich in \cite{NazarovSodin2} is for connected components of a nodal set, but as indicated in the proof it is an abstract result that holds in more generality (cf. \cite{SarnakWigman}, \cite{ALL}) that includes the case at hand.

Dividing by $|\DD_R|$, we rewrite \eqref{eq:IGS} as
\be
\left(1-\frac{r}{R}\right)^2\frac{1}{|\DD_{R-r}|}\int_{\DD_{R-r}} \frac{\N(x, r)}{|\DD_r |} dx
\leq \frac{\N(0,R)}{|\DD_R|} \leq \left(1+\frac{r}{R}\right)^2\frac{1}{|\DD_{R+r}|}\int_{\DD_{R+r}} \frac{\N^*(x, r)}{| \DD_r |} dx.
\ee 

Taking expectation 
and using translation invariance to conclude that $\EE \N(x,r)$ is independent of $x$, we obtain
\be\label{eq:sandexp}
\left(1-\frac{r}{R}\right)^2 \frac{\EE \N(0, r)}{|\DD_r |} 
\leq \frac{\EE \N(0,R)}{|\DD_R|} \leq \left(1+\frac{r}{R}\right)^2 \frac{\EE \N^*(0, r)}{| \DD_r |} .
\ee 
Next we assert that 
\be\label{eq:tangencies}
 \N^*(0,r) \leq \N(0,r) + \sT(r),
\ee
where 
$$\sT(r) = \# \left\{ (x,y) \in \partial \DD_r : \langle F(x,y),\binom{x}{y} \rangle = 0 \right\}$$ denotes the number of tangencies of $F$ with the circle of radius $r$.
Indeed, each limit cycle that intersects but is not completely contained in $\DD_r$
must have at least one entry and exit point along $\partial \DD_r$.  By considering the intersection of $\partial \DD_r$ with the interior of the limit cycle and selecting one of its connected components, we may choose such an entry-exit pair to be the endpoints of a circular arc of $\partial \DD_r$ that is completely contained in the interior of the limit cycle.  At these entry and exit points, $F$ is directed inward and outward, respectively, and by the intermediate value theorem applied to $\langle F, \binom{x}{y}\rangle$ there is an intermediate point along $\partial \DD_r$ where $F$ is tangent to $\partial \DD_r$. By our choice of the entry-exit pair, the point of tangency is in the interior of the limit cycle.  Thus, empty limit cycles correspond to distinct such points of tangency, and this verifies \eqref{eq:tangencies}.

Next we use the Kac-Rice formula to prove the following lemma concerning the expectation of the number $\sT_r$ of such tangencies along $\partial \DD_r$.

\begin{lemma}\label{lemma:tang}
The expectation of $\sT_r$ satisfies
\be\label{eq:lineartang}
\EE \sT_r = 2 \sqrt{1+r^2}.
\ee
\end{lemma}

\begin{proof}[Proof of Lemma \ref{lemma:tang}]
Let $h_1(r,\theta)$ denote 
 $\frac{x}{r}f(x,y)+\frac{y}{r}g(x,y)$
evaluated at $x=r\cos (\theta), y=r \sin (\theta)$.
For fixed $r$ the zeros of $h_1$ correspond to the tangencies counted by $\sT_r$.
Hence, applying Theorem \ref{thm:KR} 
(it is easy to check that the conditions (i)-(iv) are satisfied)
gives
\be\label{eq:KRtang}
\EE \sT_r = \int_{0}^{2\pi} \int_\RR |B| \rho_\theta(0,B) dB d\theta,
\ee
where $\rho_{\theta}(A,B)$
denotes the joint probability density
of $A=h_1(r,\theta)$ and $B=\partial_\theta h_1(r,\theta)$.
The vector field $\binom{x/r}{y/r}$ is invariant with respect to rotations about the origin,
and $f,g$ are also invariant (meaning the distribution of probability on the space of these random functions is invariant) with respect to rotations about the origin.
This implies that the inside integral in \eqref{eq:KRtang} is independent of $\theta$,
and this gives
\be\label{eq:KRtang2}
\EE \sT_r = 2\pi \int_\RR |B| \rho_0(0,B) dB .
\ee
Computation of $B = \partial_\theta h_1(r,\theta)$ evaluated at $\theta=0$ gives
$$ \partial_\theta h_1 (r,\theta) \rvert_{\theta=0} = r f_y(r,0) + g(r,0) .$$
Since $f,g$ are independent of each other, and evaluation of $f$ is independent of $f_y$ evaluated at the same point, we have that
$rf_y(r,0) + g(r,0)$ is independent of $f(r,0)$.
So the joint density $\rho_0$ of $A\rvert_{\theta=0}=f(r,0)$ and $B\rvert_{\theta=0}=rf_y(r,0)+g(r,0)$ is the product of densities
of $f(r,0) \sim N(0,1)$ 
and $rf_y(r,0)+g(r,0) \sim N(0,1+r^2)$
which gives
$$\rho_0(A,B) = \frac{1}{2\pi\sqrt{1+r^2}} \exp\left\{-\frac{A^2}{2} \right\}  \exp\left\{-\frac{B^2}{2(1+r^2)} \right\} ,$$
and in particular
\be
\rho_0(0,B) = \frac{1}{2\pi\sqrt{1+r^2}}   \exp\left\{-\frac{B^2}{2(1+r^2)} \right\} .
\ee

Then \eqref{eq:KRtang2} becomes
\begin{equation}
\EE \sT_r = \frac{2\pi}{\sqrt{2\pi}} \int_\RR |B|
\frac{1}{\sqrt{2\pi(1+r^2)}}   \exp\left\{-\frac{B^2}{2(1+r^2)} \right\} B,
\end{equation}
where we have separated the constants so that the integral gives the absolute moment of a Gaussian of mean zero and variance $1+r^2$.
From this we conclude
\begin{equation}
\EE \sT_r = 2\sqrt{1+r^2},
\end{equation}
which gives \eqref{eq:lineartang} as desired and concludes the proof of the lemma.
\end{proof}

Applying \eqref{eq:tangencies} in
\eqref{eq:sandexp} and returning to the abbreviated notation $\hat{N}_R = \N(0,R)$, we have
\be\label{eq:sandexp2}
\left(1-\frac{r}{R}\right)^2 \frac{\EE \hat{N}_r}{|\DD_r |} 
\leq \frac{\EE  \hat{N}_R}{|\DD_R|} \leq \left(1+\frac{r}{R}\right)^2 \frac{\EE  \hat{N}_r+ \EE \sT(r)}{| \DD_r |} .
\ee 

Let $\e>0$ be arbitrary.
We will show that there exists $r$ such that for all $R \gg r$ sufficiently large we have
\be\label{eq:goalrR}
\left| \frac{\EE \hat{N}_R}{|\DD_R|} -  \frac{\EE \hat{N}_r}{| \DD_r |} \right| < \e .
\ee

From \eqref{eq:lineartang}
we have 
$$ \frac{\EE \sT(r) }{ | \DD_r|} = O(r^{-1}),$$
and using this we choose $r$ large enough so that
\be\label{eq:smalltang}
 \frac{\EE \sT(r) }{ | \DD_r|} < \frac{\e}{8}.
\ee
We have
$$\left(1-\frac{r}{R} \right)^2 \frac{\EE \hat{N}_r }{ | \DD_r|} > \frac{\EE \hat{N}_r }{ | \DD_r|} - 2\frac{r}{R} \frac{\EE \hat{N}_r }{ | \DD_r|} ,$$
and we may choose $R>r$ sufficiently large (with the above choice of $r$ now fixed) so that
\be\label{eq:lower}
\left(1-\frac{r}{R} \right)^2 \frac{\EE \hat{N}_r }{ | \DD_r|} > \frac{\EE \hat{N}_r }{ | \DD_r|} - \e .
\ee
Choosing $R$ larger if necessary we also have
\be\label{eq:upper}
\left(1+\frac{r}{R} \right)^2 \frac{\EE \hat{N}_r }{ | \DD_r|} < \frac{\EE \hat{N}_r }{ | \DD_r|} + \frac{\e}{2} .
\ee

Then \eqref{eq:smalltang}, \eqref{eq:lower}, and \eqref{eq:upper} imply
\be
\frac{\EE \hat{N}_r}{|\DD_r |} - \e
< \frac{\EE \hat{N}_R}{|\DD_R|} < \frac{\EE \hat{N}_r}{| \DD_r |} + \e ,
\ee
which implies \eqref{eq:goalrR}.

It follows that 
$\frac{\EE \hat{N}_R}{|\DD_R|}$ converges as $R \rightarrow \infty$, i.e., there exists a constant $c$ such that
\begin{equation}
\EE \hat{N}_R \sim c \cdot R^2, \quad \text{as } R \rightarrow \infty.
\end{equation}
Positivity of the constant $c$ follows from a simple adaptation of the proof of Theorem \ref{thm:main}; instead of considering shrinking elliptical annuli one can take a collection of circular annuli of fixed radius.  One can fit
$\geq k\cdot R^2$ many disjoint such annuli in $\DD_R$.  The subsequent constructions simplify as well since the distortion factor $a$ is absent.
We omit further details since no new complications arise.
\end{proof}

\section{Limit cycles surrounding a perturbed center focus: Proofs of Theorems \ref{thm:as} and \ref{thm:addendum}}\label{sec:smallpert}

\begin{proof}[Proof of Theorem \ref{thm:as}]
Recall from the statement of the theorem that $N_d(\rho)$ denotes the number of limit cycles in the disk $\DD_\rho$ of the vector field $$F(x,y)=\binom{y+\e(d)p_d(x,y)}{-x + \e(d) q_d(x,y)}, $$
where
\be\label{eq:seriesp}
 p_d(x,y) = \sum_{1 \leq j+k \leq d} a_{j,k} x^j y^k,
\ee
\be\label{eq:seriesq}
q_d(x,y) = \sum_{1 \leq j+k \leq d} b_{j,k} x^j y^k.
\ee
Let $p_\infty$ and $q_\infty$
denote the random bivariate power series
obtained by letting $d\rightarrow \infty$ in \eqref{eq:seriesp}, \eqref{eq:seriesq}.

Fix $R$ satisfying $\rho < R < 1$.

Since $|a_{j,k}|, |b_{j,k}| \leq 1$
the series \eqref{eq:seriesp}, \eqref{eq:seriesq} are each majorized by
\be\label{eq:majorization}
\sum_{j,k \geq 1} |x|^j |y|^k = \sum_{j\geq 1} |x|^j \sum_{k\geq 1} |y|^k,
\ee
which converges uniformly in $\{(x,y)\in \CC^2 :|x|\leq R, |y| \leq R \}$.
In particular, $p_\infty$ and $q_\infty$ converge absolutely and uniformly in the bidisk $\DD_R \times \DD_R$.

Since our goal is to prove almost sure convergence of $N_d(\rho)$, it is important to note that  $p_d$ and $p_\infty$
are naturally coupled 
in a single probability space;
the $d$th-order truncation of $p_\infty$
is distributed as $p_d$.
Thus, in order to prove the desired almost sure convergence, we sample $p_\infty$, $q_\infty$ and then show that,
almost surely, the sequence $N_d(\rho)$ associated with the truncations $p_d$, $q_d$
converges as $d \rightarrow \infty$.

We take note of the following error estimates based on the tail of the majorization \eqref{eq:majorization},
\begin{equation}\label{eq:majortail}
    \sup_{\DD_R \times \DD_R} |p_\infty - p_d| < R^{2d}/(1-R)^2, \quad \sup_{\DD_R \times \DD_R} |q_\infty - q_d| < R^{2d}/(1-R)^2.
\end{equation}

Consider the vector field
\begin{equation}\label{eq:infinity}
F_\infty(x,y) = \binom{y + \e p_\infty(x,y)}{-x + \e q_\infty(x,y)},
\end{equation}
and let $\sP_{\infty,\e}:\RR_+\rightarrow \RR_+$ denote the corresponding Poincar\'e first return map along the positive $x$-axis.
Then by Theorem \ref{thm:PPM}, we have the following perturbation expansion
\be\label{eq:pertexp}
\sP_{\infty,\e}(r) = r + \e \sA_\infty(r) + \e^2 E(r,\e) ,
\ee
where $\sA_\infty(r)$ is the first Melnikov function
given by the integral
\begin{equation}\label{eq:A1}
    \sA_\infty(r) = \int_{x^2+y^2=r^2} p_\infty dy - q_\infty dx.
\end{equation}
Let $X(\rho)$ denote the number of zeros of $\sA_\infty$ in $(0,\rho)$.
$X(\rho)$ is our candidate limit for the sequence $N_d(\rho)$.
Before proving this convergence, let us first verify the second part of the conclusion in the theorem, namely, $\EE X(\rho)< \infty$.
Note that $\sA_\infty$ is analytic in the unit disk, and by the non-accumulation of zeros for analytic functions,
$X(\rho)$ is finite-valued.
On the other hand, the finiteness of its expectation $\EE X(\rho)$ requires additional estimates shown in the following lemma.
\begin{lemma}\label{lemma:Finite}
The expectation $\EE X(\rho)$ of the number of zeros of $\sA_\infty$ in $(0,\rho)$ satisfies
\be
\EE X(\rho) < \infty.
\ee
\end{lemma}

\begin{proof}[Proof of Lemma \ref{lemma:Finite}]
As above, fix $R$ satisfying $\rho < R < 1$.
Writing the integral \eqref{eq:A1} in polar coordinates, we have
\be\label{eq:Aint}
    \sA_\infty(r) = \int_{0}^{2\pi} p_\infty(r \cos(\theta),r \sin(\theta)) r\cos(\theta) d\theta +  q_\infty(r \cos(\theta),r \sin(\theta)) r\sin(\theta) d\theta.
\ee
This defines an analytic function valid for complex values of $r$, and for $r\in \DD_R$ with $\theta \in [0,2\pi]$ we have $(r\cos(\theta),r\sin(\theta)) \in \DD_R \times \DD_R$.
From the majorization \eqref{eq:majorization},
we have $|p_\infty|$ and $|q_\infty|$ are uniformly bounded by $1/(1-R)^2$ in $\DD_R \times \DD_R$.
Using this to estimate the above integral we obtain
\be\label{eq:Asupest}
\sup_{r \in \DD_R} |\sA_\infty(r)| \leq 2\pi \frac{2R}{(1-R)^2}.
\ee
Integrating term by term in \eqref{eq:Aint},
we can write $\sA_\infty(r)$ as a series.
\be\label{eq:Ainfinity}
    \sA_\infty(r) = \sum_{m=0}^\infty \zeta_m r^{2m+2},
\ee
where
\be\label{eq:zetam}
\zeta_m = \int_{0}^{2\pi}   \sum_{j+k=2m+1} (a_{j,k} \cos(\theta) + b_{j,k} \sin(\theta) ) (\cos(\theta))^j(\sin(\theta))^k   d\theta.
\ee
(The odd powers of $r$ do not survive in the series due to symmetry.)

Let $f(r) = \sum_{m=0}^\infty \zeta_m r^{2m} = A_\infty(r)/r^2$. 
In order to see that
$\EE X(\rho) < \infty$
we consider the number $Y(\rho)$ of complex zeros of $f$ in the disk $\DD_\rho$, and we use the following estimate based on Jensen's formula \cite[Ch. 5, Sec. 3.1]{Ahlfors}
$$Y(\rho) \leq \frac{1}{\log(\rho/R)} \log\left(\frac{M}{ |\zeta_1|}\right),$$
where
$$M:=\sup_{|r|=R} |f(r)|.$$
This gives
$$\EE Y(\rho) \leq \frac{1}{\log(\rho/R)} \left( \EE \log M - \EE \log |\zeta_1| \right) .$$
From \eqref{eq:Asupest} we have
\be
\sup_{|r|=R} |f(r)| \leq \frac{4\pi}{R(1-R)^2},
\ee 
and hence it suffices to show that
\be\label{eq:finite}
- \EE \log |\zeta_1| < \infty.
\ee
We obtain, after computing the integrals in \eqref{eq:zetam},
\be\label{eq:zeta1}
\zeta_1= 2\pi ( a_{3,0} + b_{2,1}/4 + a_{1,2}/4 + b_{0,3}).
\ee
Then the desired estimate \eqref{eq:finite} follows from the triangle inequality and finiteness of the integral $\displaystyle \int_{-1}^1 \log |u| du.$
This completes the proof of the lemma.
\end{proof}

We have that $\sA_\infty$ is smooth on $(0,\rho)$
(in fact $\sA_\infty$ is analytic in the unit disk), and the probability density of its point evaluations are bounded.
Hence, we can apply Bulinskaya's Lemma \cite[Prop. 1.20]{AzaisWscheborbook} to conclude that almost surely the zeros of $\sA_\infty$ are all non-degenerate, i.e.,
the derivative $\sA_\infty'$ does not vanish at any zero.
Non-degeneracy implies the following lemma.

\begin{lemma}\label{lemma:C1}
Almost surely, the zero set of $\sA_\infty$ is stable with respect to $C^1$-small perturbations.
\end{lemma}

Let us use the following notation for the $C^1$-norm.
$$\left\| \mathscr{F} \right\|_{C^1} := \sup_{0<r<\rho} |\sF(r)| + \sup_{0<r<\rho} |\sF'(r)| .$$

Let $\sP_d:\RR_+ \rightarrow \RR_+$ denote the Poincar\'e first return map along the positive $x$-axis of the vector field
\begin{equation}
F(x,y) = \binom{y + \e p_d(x,y)}{-x + \e q_d(x,y)}.
\end{equation}

In order to show that $N_d(\rho)$ converges to $X(\rho)$, by Lemma \ref{lemma:C1} it is enough to show that $(\sP_d(r) - r)/\e(d)$  approaches $\sA_\infty$ in the $C^1$-norm as $d \rightarrow \infty$, i.e.,
\be\label{eq:C1close2}
\lim_{d \rightarrow \infty}\left\| \frac{\sP_{d}(r) - r}{\e(d)} - \sA_\infty(r) \right\|_{C^1}=0.
\ee

Since Theorem \ref{thm:PPM} provides that the error term $E(r,\e)$ in the perturbation expansion \eqref{eq:pertexp} is uniformly bounded, we have that
$(\sP_{\infty,\e}(r)-r)/\e$ approaches $\sA_\infty$ in the $C^1$-norm as $\e \rightarrow 0$.
In particular, since $\e(d) \rightarrow 0$ as $d \rightarrow \infty$, we have
\be\label{eq:C1close}
\lim_{d \rightarrow \infty}\left\| \frac{\sP_{\infty,\e(d)}(r) - r}{\e(d)} - \sA_\infty(r) \right\|_{C^1}=0.
\ee

We have
\be\label{eq:reduced}
\left\| \frac{\sP_d(r) - r}{\e(d)} - \sA_\infty(r) \right\|_{C^1} \leq \left\| \frac{\sP_{\infty,\e(d)} - \sP_d}{\e(d)} \right\|_{C^1} + \left\| \frac{\sP_{\infty,\e(d)}(r) - r}{\e(d)} - \sA_\infty(r) \right\|_{C^1},\ee
and, as we have noted in \eqref{eq:C1close}, the second term on the right hand side of \eqref{eq:reduced} converges to zero as $d\rightarrow \infty$.
Hence, our immediate goal \eqref{eq:C1close2} is reduced to proving the following lemma.
\begin{lemma}\label{lemma:delta}
Let $\delta>0$ be arbitrary. 
For all $d$ sufficiently large, we have
\be\label{eq:suff} \left\| \frac{\sP_{\infty,\e(d)} - \sP_d}{\e(d)} \right\|_{C^1} < \delta. \ee
\end{lemma}

The following proof could most likely be condensed with an appropriate estimate, which surely must be known, for the variation of the Poincar\'e map under $C^1$-small perturbations of the underlying vector field.  However, we were unable to find a suitable reference, so we give the following mostly self-contained proof.

\begin{proof}[Proof of Lemma \ref{lemma:delta}]
In order to show this, we first change to polar coordinates $x=r \cos \theta, y = r \sin \theta$ and write our systems $\binom{\dot{x}}{\dot{y}} = \binom{y+\e(d) p_d(x,y)}{-x + \e(d) q_d(x,y)}$ and $\binom{\dot{x}}{\dot{y}} = \binom{y+\e p_\infty(x,y)}{-x + \e q_\infty(x,y)}$ as
\begin{equation}\label{eq:dH_d}
\frac{d r}{d \theta} = \e(d)H_d(r,\theta),
\end{equation}
and
\begin{equation}\label{eq:dH_inf}
\frac{d r}{d \theta} = \e H_{\infty,\e}(r,\theta),
\end{equation}
where
$$H_d(r,\theta) = \frac{(x p_d(x,y)+y  q_d(x,y))/r}{1+\e(d)(x q_d(x,y)-y p_d(x,y))/r^2},$$
and
$$H_{\infty,\e}(r,\theta) = \frac{(x p_\infty(x,y)+y  q_\infty(x,y))/r}{1+\e (x q_\infty(x,y)-y p_\infty(x,y))/r^2}.$$

Fix an initial condition $r(0) = r_0$. We will use $r_d(\theta)$ and $r_\infty(\theta)$ to denote the solutions to the systems \eqref{eq:dH_d}
and \eqref{eq:dH_inf}, respectively, with the same initial condition $r(0)=r_0$.

The estimates \eqref{eq:majorization}, \eqref{eq:majortail} imply that
given any $\delta_0$, we have for all $d$ sufficiently large
\be \label{eq:delta0}
\sup_{(r,\theta) \in \DD_R \times [0,2\pi]} \left| H_{\infty,\e(d)}(r,\theta) - H_d(r,\theta) \right| < \delta_0.
\ee
We also have from the estimates \eqref{eq:majortail} that there exists $M>0$ such that
\be\label{eq:M}
\sup_{(r,\theta) \in \DD_R \times [0,2\pi]} |H_d(r,\theta)| \leq M , \quad \sup_{(r,\theta) \in \DD_R \times [0,2\pi]} |H_{\infty,\e(d)}(r,\theta)| \leq M.
\ee

We have
\be 
\frac{d}{d \theta} \left( r_\infty(\theta)-r_d(\theta) \right)
= \e(d) \left[ H_{\infty,\e(d)}(r_\infty(\theta),\theta)- H_d(r_d(\theta),\theta) \right],
\ee
which implies
\be \label{eq:intEst}
|r_\infty(\theta)-r_d(\theta)| \leq \e(d) \int_{0}^{\theta} \left| H_{\infty,\e(d)}(r_\infty(\theta),\theta)- H_d(r_d(\theta),\theta) \right| d\theta.
\ee

Applying \eqref{eq:M} to estimate \eqref{eq:intEst} we obtain
\be \label{eq:firstEst}
|r_d(\theta) - r_\infty(\theta)| < 2 M \e(d).
\ee

We have $\sP_{\infty,\e(d)}(r_0) = r_\infty(2\pi)$, $\sP_d(r_0) = r_d(2\pi)$,
where recall $r_\infty, r_d$ are the solutions to the systems \eqref{eq:dH_inf}, \eqref{eq:dH_d} with initial condition $r_\infty(0) = r_d(0)=r_0$.
Using this while setting $\theta = 2\pi$ in \eqref{eq:intEst} we have
\be \label{eq:intEst2pi}
|\sP_{\infty,\e(d)}(r_0)-\sP_d(r_0)| \leq \e(d) \int_{0}^{2\pi} \underbrace{\left| H_{\infty,\e(d)}(r_\infty(\theta),\theta)- H_d(r_d(\theta),\theta) \right|}_{(*)} d\theta.
\ee
The above integrand $(*)$ satisfies
\be\label{eq:*}
(*) \leq \left| H_d(r_d(\theta),\theta) - H_{\infty,\e(d)}(r_d(\theta),\theta) \right| + \left| H_{\infty,\e(d)}(r_d(\theta),\theta)-H_{\infty,\e(d)}(r_\infty(\theta),\theta)\right|.
\ee 
We have from \eqref{eq:delta0}
\be\label{eq:*part1}
| H_d(r_d(\theta),\theta) - H_{\infty,\e(d)}(r_d(\theta),\theta) | < \delta_0, \quad \theta \in [0,2\pi],
\ee
and we also have for all $d$ sufficiently large
\be\label{eq:*part2}
|H_{\infty,\e(d)}(r_d(\theta),\theta)-H_{\infty,\e(d)}(r_\infty(\theta),\theta)| <  \delta_0, \quad \theta \in [0,2\pi].
\ee
The latter follows from the
uniform equicontinuity of the sequence $H_{\infty,\e(d)}$
since we have as stated in \eqref{eq:firstEst}
that $|r_\infty(\theta)-r_d(\theta)|$
is arbitrarily small for $d$ sufficiently large.

Using \eqref{eq:*}, \eqref{eq:*part1}, \eqref{eq:*part2} to estimate \eqref{eq:intEst2pi}, we obtain
\be
|\sP_{\infty,\e(d)}(r_0)-\sP_d(r_0)| \leq \e(d) 4\pi \delta_0.
\ee

This gives
\be\label{eq:supPoincare}
\sup_{r\in \DD_R} \left| \frac{\sP_{\infty,\e(d)}(r)-\sP_d(r)}{\e(d)} \right| < 4\pi \delta_0,
\ee
and using Cauchy estimates \cite[Ch. 4, Sec. 2.3]{Ahlfors} for the derivatives
we obtain
\be
\sup_{r\in \DD_\rho} \left| \frac{\sP'_{\infty,\e(d)}(r)-\sP'_d(r)}{\e(d)} \right| < \frac{4\pi}{R-\rho} \delta_0,
\ee
and together with \eqref{eq:supPoincare} this gives
\be
\left\| \frac{\sP_{\infty,\e(d)}(r)-\sP_d(r)}{\e(d)} \right\|_{C^1} < 4\pi \left( 1+ \frac{1}{R-\rho}\right) \delta_0.
\ee
Since $\delta_0$ was arbitrary, this verifies
\eqref{eq:suff} and concludes the proof of the lemma.
\end{proof}

Lemma \ref{lemma:delta} along with
\eqref{eq:reduced} and \eqref{eq:C1close}
implies \eqref{eq:C1close2}.
Hence, almost surely, we have for all $d$ sufficiently large, $(\sP_d(r)-r)/\e(d)$
lies  within the $C^1$-neighborhood of stability of $\sA_\infty$ provided by Lemma \ref{lemma:C1}, and we conclude that the fixed points of $\sP_d$ in the interval $(0,\rho)$ are in one-to-one correspondence with the zeros of $\sA_\infty$.
This verifies the desired almost sure convergence of $N_d(\rho)$ to $X(\rho)$.

Finally, we verify the statements in the theorem concerning the random coefficients $\zeta_m$ of the series $\sA_\infty(r)$.
From \eqref{eq:zetam} we see that each $\zeta_m$ is a linear combination 
\be
\zeta_m =  \sum_{j+k=2m+1} a_{j,k}\int_{0}^{2\pi}    (\cos(\theta))^{j+1}(\sin(\theta))^k   d\theta + b_{j,k}\int_{0}^{2\pi}    (\cos(\theta))^j(\sin(\theta))^{k+1}  d\theta
\ee
of the random variables $a_{j,k}, b_{j,k}$, and as such each $\zeta_m$ has mean zero.
It is easy to see from the range of the indices of the sum in \eqref{eq:zetam} that $\zeta_m$ are independent since the $a_{j,k},b_{j,k}$ appearing in \eqref{eq:zetam} for distinct $m$ do not overlap.
Recalling that $\EE a_{j,k}^2 = \EE b_{j,k}^2 = 1$ for each $j,k$, the variance $\sigma_m^2 = \EE \zeta_m^2 $ of $\zeta_m$ is given by
\be
\label{eq:sigma_m}
\sigma_m^2 = \sum_{k=0}^{2m+1} \left( \int_{0}^{2\pi}    (\cos(\theta))^{2m+2-k}(\sin(\theta))^k   d\theta\right)^2  +  \left( \int_{0}^{2\pi}    (\cos(\theta))^{2m+1-k}(\sin(\theta))^{k+1}   d\theta\right)^2 .
\ee
We use the identities \cite{Grad}
$$ \int_{0}^{2\pi}    (\cos(\theta))^{2m+2-k}(\sin(\theta))^k   d\theta =
\begin{cases} 
2\pi \frac{(2m-k+1)!!(k-1)!!}{(2m+2)!!} &, \quad k \text{ even} \\
0 &,  \quad k \text{ odd}
\end{cases},$$
$$ \int_{0}^{2\pi}    (\cos(\theta))^{2m+1-k}(\sin(\theta))^{k+1}   d\theta =
\begin{cases} 
 0 &, \quad k \text{ even} \\
2\pi \frac{(2m-k)!!(k)!!}{(2m+2)!!} &,  \quad k \text{ odd}
\end{cases},$$
in order to obtain
\be\sigma_m^2 = 4\pi^2 \sum_{\ell=0}^m
\left(\frac{(2m-2\ell+1)!!(2\ell-1)!!}{(2m+2)!!}\right)^2 + 
 \left(\frac{(2m-2\ell-1)!!(2\ell+1)!!}{(2m+2)!!}\right)^2 .
\ee
The asymptotic behavior as $m\rightarrow \infty$
is dominated by just two terms (one for $\ell=0$ and one for $\ell=m$), and
we have
\be\label{eq:asymvar}
\sigma_m^2 \sim 8\pi^2 \left( \frac{(2m+1)!!}{(2m+2)!!}\right)^2 \sim  8\pi m^{-1},
\ee
as stated in the theorem.
\end{proof}

\begin{proof}[Proof of Theorem \ref{thm:addendum}]
We shall reduce this problem to one where we can apply the results from \cite{Brudnyi1}.  Some lines below closely follow the proof of \cite[Thm. A]{Brudnyi1}.
We have 
\be
 p(x,y) = \sum_{1 \leq j+k \leq d} a_{j,k} x^j y^k, \quad 
q(x,y) = \sum_{1 \leq j+k \leq d} b_{j,k} x^j y^k,
\ee
with the vector of all coefficients
sampled uniformly from the $d(d+3)$-dimensional unit ball $\displaystyle \left\{ \sum_{1 \leq j+k \leq d} (a_{j,k})^2 + (b_{j,k})^2 \leq 1 \right\}$.
We change variables by a scaling
$u=x/(2\rho)$, $v=y/(2\rho)$.
In these coordinates, the system
\begin{equation}\label{eq:xy}
\binom{\xdot}{\ydot} = \binom{-y + \e p(x,y)}{x + \e q(x,y)}
\end{equation}
becomes
\begin{equation}\label{eq:uv}
\binom{\udot}{\vdot} = \binom{- v + \e \hp(u,v)}{ u + \e \hq(u,v)},
\end{equation}
where
$\hp(u,v) = 2 \rho \cdot p(2\rho u, 2 \rho v)$,
and $\hq(u,v) = 2 \rho \cdot  q(2\rho u, 2 \rho v)$.

We are concerned with limit cycles
situated in the disk $\DD_{1/2}$ of the $(u,v)$-plane.

Changing to polar coordinates $u=r \cos \phi, v = r \sin \phi$, we obtain
\begin{equation}\label{eq:dH}
\frac{\p r}{\p \phi} = H(r,\phi) r, \quad H(r,\phi) = \frac{\e(u\hp(u,v)+v \hq(u,v))/r^2}{1+\e(u\hq(u,v)-v\hp(u,v))/r^2}.
\end{equation}
We complexify the radial coordinate $r$,
and let $U = \DD \times [0,2\pi]$.
Then a simple modification of the estimates in \cite[p. 236]{Brudnyi1} gives
$$ \sup_{U} \left| \frac{\hp(u,v)}{r} \right| \leq (2\rho)^d \sqrt{\sum (a_{j,k})^2} \sqrt{d} \leq (2\rho)^d \sqrt{d} ,$$
$$ \sup_{U} \left| \frac{\hq(u,v)}{r} \right| \leq (2\rho)^d \sqrt{\sum (b_{j,k})^2} \sqrt{d}
\leq (2\rho)^d \sqrt{d},$$
where we used $\sum (a_{j,k})^2 + (b_{j,k})^2 \leq 1$.
These estimates imply
\begin{align*}
\sup_{U} | H(r,\phi)| &\leq \frac{\e (2\rho)^d\sqrt{d}}{1-\e (2\rho)^d \sqrt{d}} \\
&\leq 3 \sqrt{d} (2\rho)^d \e.
\end{align*}
Since $\e \leq \frac{(2\rho)^{-d(d+3)}}{40\pi \sqrt{d}}$,
we have
$$ \sup_{U} | H(r,\phi)| \leq \frac{3}{40\pi},$$
so that the hypothesis of \cite[Prop. 3.1]{Brudnyi1} is satisfied
and we conclude that
the Poincar\'e map $p_\mu$ as in \cite[p. 237]{Brudnyi1}
along with the functions
$$g_\mu(r) = \frac{p_\mu(r)}{r} - 1, \quad h_\mu(r) = \frac{40 \sqrt{d}}{\sqrt{2}}g_\mu(r) $$
are analytic in $\DD_{3/4}$
and depend anaytically on the vector $\mu$ of coefficients throughout the ball
in $\CC^{d(d+3)}$ of radius $2N$, with $N = \frac{1}{40\pi \sqrt{d}}$.
The rest of the proof, consisting of an application of \cite[Thm. 2.3]{Brudnyi1} now 
follows exactly as in \cite[p. 237]{Brudnyi1}.
\end{proof}

\section{Infinitesimal perturbations: Proofs of Theorems \ref{thm:sqrt} and \ref{thm:powerlaw}}\label{sec:infinitesimal}

Before proving Theorem \ref{thm:sqrt}, we state several lemmas that hold more generally for perturbed Hamiltonian systems.

Consider the system
\be\label{eq:systemKostpertHam}
\begin{cases}
\dot{x} = \partial_x H(x,y) + \e p(x,y) \\
\dot{y} = -\partial_y H(x,y) + \e q(x,y)
\end{cases},
\ee
 with $H$ a generic (fixed) polynomial and $p,q$ Kostlan random polynomials of degree $d$.
 Let $A$ be a period annulus of the unperturbed Hamiltonian system ($\e=0$), and let $N_{d,\e}(A)$ denote the number of limit cycles in $A$.

\begin{lemma}\label{lemma:Nd}
Fix a period annulus $A$ of the Hamiltonian system.  As $\e \rightarrow 0^+$ the number $N_{d,\e}(A)$ of limit cycles in $A$ converges almost surely to the random variable $N_d(A)$ that counts the number of zeros of the first Melnikov function for the system \eqref{eq:systemKostpertHam} given by
\be\label{eq:firstMelnikov}
\sA(t) = \int_{C_t} p dy - q dx,
\ee
where $C_t = \{H(x,y)=t\}$.
\end{lemma}

\begin{proof}[Proof of Lemma \ref{lemma:Nd}]
Since $p,q$ are Gaussian random polynomials
and the curves $C_t$ provide a smooth foliation of $A$, the random function $\sA$ is a smooth (in fact analytic) non-degenerate Gaussian random function, and the probability density of its point evaluations $\sA(t)$ are uniformly bounded.
Bulinskaya's Lemma \cite[Prop. 1.20]{AzaisWscheborbook} then implies
that almost surely the zeros of $\sA$ are all non-degenerate;
the lemma now follows from an application of Theorem \ref{thm:PPM}.
\end{proof}

\begin{lemma}\label{lemma:EK}
The expectation $\EE N_d(A) $ of the number of zeros of $\sA$ in $A$ is given by
\begin{equation}\label{eq:EK}
\EE N_d(\rho) = \frac{1}{\pi} \int_a^b \sqrt{ \frac{\partial^2}{\partial r \partial t} \log(\sK(r,t)) \rvert_{r=t=\tau} } \, d\tau,
\end{equation}
where $a$ and $b$ are chosen so that $A$ is a component of the set $\{ a < H(x,y) < b \}$ and where
\be\label{eq:2pt}
\sK(r,t) = \EE \sA(r) \sA(t)
\ee
denotes the two-point correlation function of $\sA$.
\end{lemma}

\begin{proof}[Proof of Lemma \ref{lemma:EK}]
Since $\sA$ satisfies the conditions of Theorem \ref{thm:KR}, this follows from an application of Edelman and Kostlan's formulation \eqref{eq:EKform} of the Kac-Rice formula.
\end{proof}

\begin{lemma}\label{lemma:twopt}
The two-point correlation function of $\sA$ defined in \eqref{eq:2pt} satisfies
\be \label{eq:kernel}
\sK(r,t)= \int_{C_r} \int_{C_t}  (1+x_1 x_2 +y_1 y_2)^d \left[ dy_1 dy_2 + dx_1 dx_2 \right].
\ee
\end{lemma}

\begin{proof}[Proof of Lemma \ref{lemma:twopt}]

Let us write
\be\label{eq:ArAt}
\sA(r) = \int_{C_r} p(v_1) dy_1 - q(v_1)dx_1, \quad \sA(t) = \int_{C_r} p(v_2) dy_2 - q(v_2)dx_2,
\ee 
with $v_1=(x_1,y_1), v_2=(x_2,y_2)$.
Using \eqref{eq:ArAt}
and linearity of expectation,
the expectation in 
\eqref{eq:2pt} can be expressed as
\be
\int_{C_r} \int_{C_t} \EE \left[ p(v_1) p(v_2) dy_1 dy_2 - p(v_1)q(v_2)dy_1 dx_2 - q(v_1)p(v_2)dx_1 dy_2 + q(v_1)q(v_2) dx_1 dx_2 \right],
\ee
and since the pointwise evaluations
of $p, q$ are centered Gaussians independent of eachother,
we have, for each $v_1=(x_1,y_1), v_2=(x_2,y_2)$,
$$\EE p(v_1)q(v_2) = \EE q(v_1)p(v_2) = 0,$$
which leads to the following simplified expression for the two-point correlation.
\be
 \sK(r,t) =  \int_{C_r} \int_{C_t} \left[ \EE p(v_1) p(v_2) dy_1 dy_2 + \EE q(v_1)q(v_2) dx_1 dx_2 \right].
\ee
We write this as
\be\label{eq:covar}
     \sK(r,t) =  \int_{C_r} \int_{C_t} \left[ K(v_1, v_2) dy_1 dy_2 + K(v_1,v_2) dx_1 dx_2 \right],
\ee
where $K(v_1,v_2) = \EE p(v_1) p(v_2)$ denotes the covariance kernel for the Kostlan ensemble.
Finally, substitute in \eqref{eq:covar} the known expression $$ K(v_1,v_2) = (1+x_1 x_2 + y_1 y_2)^d$$
from \eqref{eq:2ptK}
for the covariance kernel of the Kostlan ensemble.
This gives \eqref{eq:kernel} and concludes the proof of the lemma.
\end{proof}

\begin{proof}[Proof of Theorem \ref{thm:sqrt}]
As in the statement of the theorem, $N_{d,\e}(\rho)$ denotes the number of limit cycles in $\DD_\rho$ of the system
\be\label{eq:systemKostpert}
\begin{cases}
\dot{x} = y + \e p(x,y) \\
\dot{y} = -x + \e q(x,y)
\end{cases},
\ee
 with $p,q$ Kostlan random polynomials of degree $d$.
 By Lemma \ref{lemma:Nd}, as $\e \rightarrow 0^+$ $N_{d,\e}(\rho)$ converges almost surely to the number $N_d(\rho)$ of zeros of $\sA$ in $(0,\rho)$ with
 $$\sA(r) = \int_{C_r} p dx - q dy, \quad C_r = \{ x^2 + y^2 = r^2 \} .$$



Let $\sK(r,t) =  \EE \sA(r) \sA(t)$ denote the two-point correlation function of $\sA$ as in \eqref{eq:2pt}.

We parameterize $C_r = \{ x^2 + y^2 = r^2 \}$ by $(x_1,y_1) = (r \cos \theta_1, r \sin \theta_1)$ and
$C_t$ by $(x_2,y_2) = (t \cos \theta_2, t \sin \theta_2)$.
In terms of this parameterization we can write the integral \eqref{eq:kernel} as
$$    \sK(r,t) = \int_{0}^{2\pi} \int_{0}^{2\pi}  (1+rt\cos(\theta_1-\theta_2))^d r t \cos(\theta_1-\theta_2) d\theta_1 d\theta_2.
$$
Changing variables in the inside integral with $u=\theta_1-\theta_2$, $du=d\theta_1$ leads to an integrand independent of $\theta_2$, and we obtain
\begin{equation}\label{eq:du}
   \sK(r,t) = 2 \pi \int_{0}^{2\pi}  (1+rt\cos(u))^d r t \cos(u) du.
\end{equation}

Applying Laplace's method \cite[Sec. 4.2]{deBruijn} for asymptotic evaluation of the integral \eqref{eq:du}, we find
\be\label{eq:Laplace}
    \sK(r,t) = 2\pi (1+rt)^d \sqrt{\frac{2\pi rt (rt+1)}{d}}(1+E_d(r,t)), \quad \text{as } d \rightarrow \infty,
\ee
where $E_d(r,t) = O(d^{-1})$.

From \eqref{eq:Laplace} we obtain
\be\label{eq:1/dasymp}
\frac{\log \sK(r,t)}{d} = \log(1+rt) + \frac{1}{d} \log 2\pi \sqrt{\frac{2\pi rt(1+rt)}{d}} + \frac{1}{d}\log(1+E_d(r,t)).
\ee 
The functions $\frac{\log \sK(r,t)}{d}$
are analytic in a complex neighborhood $U$ of $(r,t) \in [0,\infty) \times [0,\infty)$
and converge to $\log(1+rt)$ as $d \rightarrow \infty$.
The convergence is uniform on compact subsets of $U$. This justifies, by way of Cauchy estimates, differentiation of the asymptotic \eqref{eq:1/dasymp} to obtain
\be \label{eq:asymp}
\lim_{d \rightarrow \infty} \frac{1}{d}\frac{\partial^2}{\partial r \partial t}\log \sK(r,t) \rvert_{r=t=\tau} = \frac{1}{(1+\tau^2)^2},
\ee
where the convergence is uniform for $\tau \in [0,\rho]$.

Applying this to \eqref{eq:EK} gives
\be\label{eq:sqrtLaw}
     \EE N_d(\rho) \sim \frac{\sqrt{d}}{\pi} \int_0^\rho \frac{1}{(1+\tau^2)} \, d\tau, \quad \text{as } d \rightarrow \infty,
\ee
and computing the integral in \eqref{eq:sqrtLaw} we find
\be
     \EE N_d(\rho) \sim \frac{\sqrt{d}}{\pi} \arctan(\rho), \quad \text{as } d \rightarrow \infty,
\ee
as desired. This concludes the proof of the theorem.
\end{proof}

\begin{proof}[Proof of Theorem \ref{thm:powerlaw}]
The initial step follows the above proof of Theorem \ref{thm:sqrt}. Let $\sA$ denote the first Melnikov function of the system $\binom{\xdot}{\ydot} = F(x,y)$, where recall 
$$\displaystyle F(x,y) = \binom{y+\e p(x,y)}{-x + \e q(x,y)} $$ with
\be
 p(x,y) = \sum_{m=1}^d \sum_{j+k = m} c_m a_{j,k} x^j y^k, \quad
q(x,y) = \sum_{m=1}^d \sum_{j+k = m} c_m b_{j,k} x^j y^k,
\ee
where the deterministic weights $c_m$ satisfy \eqref{eq:powerlawvar}.

From another application of 
Lemma \ref{lemma:Nd} we have that $N_{d,\e}$
converges almost surely to $N_d$ as $\e \rightarrow 0^+$, where $N_d$ denotes the number of zeros of $\sA$.

Using polar coordinates and integrating term by term, we obtain
\begin{align*}
    \sA(r) &= \int_{x^2+y^2=r^2} p dy - q dx \\
    &= \int_{0}^{2\pi} p(r \cos(\theta),r \sin(\theta)) r\cos(\theta) d\theta +  q(r \cos(\theta),r \sin(\theta)) r\sin(\theta) d\theta \\
    &= \sum_{m=0}^{\lfloor (d-1)/2 \rfloor} c_{2m+1} \zeta_m r^{2m+2},
\end{align*}
where
\begin{align}
\zeta_m &= \int_{0}^{2\pi}   \sum_{j+k=2m+1} (a_{j,k} \cos(\theta) + b_{j,k} \sin(\theta) ) (\cos(\theta))^j(\sin(\theta))^k   d\theta \\
 &=  \sum_{j+k=2m+1} a_{j,k}\int_{0}^{2\pi}    (\cos(\theta))^{j+1}(\sin(\theta))^k   d\theta + b_{j,k}\int_{0}^{2\pi}    (\cos(\theta))^j(\sin(\theta))^{k+1}  d\theta.
\end{align}
Recalling that $\EE a_{j,k}^2 = \EE b_{j,k}^2 = 1$ for each $j,k$, we find that the variance $\sigma_m^2 = \EE \zeta_m^2 $ of $\zeta_m$ satisfies the asymptotic
$$\sigma_m^2 \sim 8\pi^2 \left( \frac{(2m+1)!!}{(2m+2)!!}\right)^2 \sim  8\pi m^{-1},$$
which follows from the same steps that led to \eqref{eq:asymvar}.
Let us write $\zeta_m = \sigma_m \hat{\zeta}_m$, where $\hat{\zeta}_m$ has mean zero, unit variance, and uniformly bounded moments, and
\be
 \sA(r) = \sum_{m=0}^{\lfloor (d-1)/2 \rfloor} c_{2m+1} \sigma_m \hat{\zeta}_m r^{2m+2}.
\ee
Letting
\be
f(s) = \sum_{m=0}^{\lfloor (d-1)/2 \rfloor} c_{2m+1} \sigma_m \hat{\zeta}_m s^{m},
\ee
we have $\sA(r) = r^2 f(r^2)$.
Note that $\sA$ and $f$ have the same number of zeros in $(0,\infty)$.
Since $c_{2m+1}^2 \sigma_m^2 \sim 8\pi 2^\gamma m^{\gamma-1}$ as $m \rightarrow \infty$, we can apply \cite[Corollary 1.6]{DoVu} to conclude
\be 
\EE N_d \sim \frac{1+\sqrt{\gamma}}{2 \pi} \log d, \quad \text{as } d \rightarrow \infty,
\ee
as desired.
\end{proof}

\section{Future Directions and Open Problems}\label{sec:concl}

\subsection{The Real Fubini-Study Ensemble}
The Kostlan model is sometimes referred to as the ``Complex Fubini-Study model'', since it arises
from the inner product associated to integration with respect to the complex Fubini-Study metric.
The Gaussian model induced by the inner product alternatively associated with integration with respect to the real Fubini-Study metric has been referred to as the ``Real Fubini-Study model'' \cite{sarnak}.

While an explicit description of the Real Fubini-Study model is more complicated than that of the Complex Fubini-Study model and requires expansions in terms of Legendre polynomials (or more generally Gegenbauer polynomials in higher-dimensions), it has the attractive feature of exhibiting more extreme behavior.
For instance, the average number of equilibria grows quadratically, and probabilistic studies on the first part of Hilbert's sixteenth problem \cite{LLstatistics}, \cite{NazarovSodin2}, \cite{SarnakWigman} show that the number of ovals in a random curve given by the zero set of a random polynomial sampled from the Real Fubini-Study model grows quadratically which is the maximal rate of growth as dictated by the Harnack curve theorem.

It seems likely that the proof of Theorem \ref{thm:main} can be extended to prove a quadratic lower bound on the average number of limit cycles of a random vector field with components sampled from the Real Fubini-Study model.
However, the lower bounds shown in \cite{ChLl} grow faster than quadratically, having an additional factor of $\log d$.
Does the average (over the Real Fubini-Study ensemble) grow faster than quadratically?

\subsection{Zeros of Random Abelian Integrals}\label{sec:abelian}
As indicated in the lemmas of Section \ref{sec:infinitesimal}, some of the methods in the proof of Theorem \ref{thm:sqrt} apply more generally to the study of perturbed Hamiltonian systems
\be
\begin{cases}
\dot{x} = \partial_y H(x,y) + \e p(x,y) \\
\dot{y} = -\partial_x H(x,y) + \e q(x,y)
\end{cases},
\ee
where $H$ is a fixed generic Hamiltonian and $p,q$ are Kostlan random polynomials of degree $d$.
The associated first Melnikov function
\be
\sA(t) = \int_{C_t} p \, dy - q \, dx
\ee
for this problem is an Abelian integral.
Here recall $C_t$ is a connected component of the level set $\{ (x,y) \in \RR^2: H(x,y) = t \}$.
Since the level sets of $H$ can have multiple connected components, the function $\sA$ is multi-valued, and we are interested in the zeros of its real branches.
We can apply Lemma \ref{lemma:EK} along each (real) branch and collect the results to obtain an exact formula for the expected number of real zeros of $\sA$.
Asymptotic analysis, on the other hand, is delicate
and will be carried out in a forthcoming work.

\subsection{Random Li\'enard Systems}
Smale posed the problem \cite{Smale1998} of estimating the number of limit cycles for the special class of vector fields
\begin{equation}\label{eq:Lienard}
F(x,y) = \binom{y - f(x)}{-x},
\end{equation}
where $f(x)$ is a real polynomial of odd degree $2k+1$ and satisfying $f(0)=0$.

Without requiring a smallness assumption on $f$, the trajectories of \eqref{eq:Lienard} retain the same topological structure as the perturbed center focus; there is a single equilibrium at the origin and the trajectories wind around this point.  In particular, the Poincar\'e map is always globally defined along the entire positive $x$-axis.

For this reason, the following probabilistic version of Smale's problem seems to provide a non-perturbative problem that is more approachable (perhaps using methods of \cite{Brudnyi1}, \cite{Brudnyi2}) than the one mentioned at the end of Section \ref{sec:KSS}.

\begin{prob}
Determine an upper bound on the expected global number of limit cycles of the vector field \eqref{eq:Lienard}
where $f$ is a random univariate polynomial.
\end{prob}

The outcome will depend on the choice of model from which $f$ is sampled.

\subsection{Limit cycles on a cylinder}
A problem of Pugh,
with a revision suggested by Lins-Neto \cite{LN}, asks to study the number of solutions of the one-dimensional differential equation with boundary condition
\be\label{eq:cylinder}
\frac{dx}{dt} = f(t,x), \quad x(0)=x(1) ,\ee
with $f$ a polynomial in $x$ whose coefficients are analytic $1$-periodic functions in $t$.
Pugh's original problem asked for an upper bound depending only on the degree in $x$. After presenting counter-examples, Lins-Neto proposed to take coefficients that are trigonometric polynomials in $t$ and asked for an upper bound depending on both the degrees in $x$ and $t$.

Solutions of \eqref{eq:cylinder} can be viewed as limit cycles on a cylinder.

We pose a randomized version of the problem where we take 
\be
f(t,x)=\sum_{\lambda_i \leq \Lambda} a_i \phi_i(t,x)
\ee
to be a Gaussian band-limited function, i.e., a truncated eigenfunction expansion.  The basis $\{\phi_i\}$ consists of Laplace eigenfunctions of the cylinder $\phi_i$ with eigenvalue $\lambda_i$, and we consider the frequency cut-off $\Lambda$ in place of degree as the large parameter in this model. 
The random function $f$ is periodic in $t$ and translation-invariant in $x$.

\begin{prob}
Let $f(t,x)$ be a Gaussian band-limited function with frequency cut-off $\Lambda$.
Study the average number $\EE N_\Lambda([a,b])$ of solutions of \eqref{eq:cylinder} over a finite interval $x(0) \in [a.b]$.
\end{prob}

Assuming finiteness of this average number, the translation-invariance in $x$ implies that the ``first intensity'' of limit cycles is constant, i.e., we have
\be
\EE N_\Lambda([a,b]) = (b-a) F(\Lambda).
\ee
Based on the natural scale provided by the wavelength $1/\Lambda$, a na\"ive guess is that $F(\Lambda) \sim C\cdot\Lambda$ as $\Lambda \rightarrow \infty$, for some constant $C>0$.

\subsection{High-dimensional Systems and the May-Wigner Instability}
Studies in population ecology indicate that an ecological system with a large number of species
is unlikely to admit stable equilibria.
This phenomenon is referred to as the  \emph{May-Wigner instability} and was first observed by May
\cite{MAY1972} in linear settings to be a consequence of Wigner's semi-circle law from random matrix theory.  An updated and more refined treatment comparing several random matrix models is provided in \cite{Allesina2012}, and
a non-linear version of the May-Wigner instability (yet still based on random matrix theory) is presented in
\cite{Fyodorov}.
All of the results in this area concern the local stability of equilibria. However, as pointed out in \cite{Allesina2015}, a lack of stable equilibria does not preclude the persistence of a system; the coexistence of populations can be achieved through the existence of stable limit cycles and other
stable invariant sets.

This naturally leads to the following problem.

\begin{prob}
Study the existence of stable limit cycles and other stable structures
in high-dimensional random vector fields.
Find an asymptotic lower bound on the probability that at least one stable forward-invariant set exists as the dimension becomes large.
\end{prob}

It may be useful to adapt elements from the proof of Theorem \ref{thm:main},
while replacing the transverse annulus with an appropriate trapping region.
However, additional tools will be required to address the high-dimensional nature of the problem.

\bibliographystyle{abbrv}
\bibliography{H16}

\end{document}